\documentclass[a4paper, 11pt,reqno]{amsart}

\usepackage{amsmath, amssymb, graphicx}

\usepackage{amsthm}
\usepackage{amsfonts}
\usepackage{mathrsfs}
\usepackage{wasysym}
\usepackage{tikz-cd}
\usepackage{mathtools}

\usepackage{multirow}
\usepackage{multicol}

\usepackage{adjustbox}

\usepackage[OT2,OT1]{fontenc}
\newcommand\cyr
{
\renewcommand\rmdefault{wncyr}
\renewcommand\sfdefault{wncyss}
\renewcommand\encodingdefault{OT2}
\normalfont
\selectfont
}

\usepackage{xcolor}
\definecolor{darkgreen}{rgb}{0,0.50,0} 
\definecolor{darkred}{rgb}{0.55,0,0}
\definecolor{darkblue}{rgb}{0,0,0.6}
\definecolor{darkteal}{rgb}{0.0, 0.25, 0.5}
\usepackage[pdfborder={0 0 0},pagebackref,colorlinks,citecolor=darkgreen,linkcolor=darkteal,urlcolor=darkblue]{hyperref}

\usepackage{cleveref}
\crefname{appsec}{Appendix}{Appendices}
\crefname{section}{Section}{Sections}
\crefname{figure}{Figure}{Figures}

\theoremstyle{plain}
    \newtheorem{theorem}{Theorem}[section]
        \crefname{theorem}{Theorem}{Theorems}
    \newtheorem*{theorem*}{Theorem}
    \newtheorem{corollary}[theorem]{Corollary}
        \crefname{corollary}{Corollary}{Corollaries}
    \newtheorem*{corollary*}{Corollary}
    \newtheorem{proposition}[theorem]{Proposition}
        \crefname{proposition}{Proposition}{Propositions}
    \newtheorem*{proposition*}{Proposition}
    \newtheorem{lemma}[theorem]{Lemma}
        \crefname{lemma}{Lemma}{Lemmas}

\theoremstyle{definition}
    \newtheorem{definition}[theorem]{Definition}
        \crefname{definition}{Definition}{Definitions}
    \newtheorem{remark}[theorem]{Remark}
        \crefname{remark}{Remark}{Remarks}
    \newtheorem{example}[theorem]{Example}
        \crefname{example}{Example}{Examples}
    
        \crefname{notation}{Notation}{Notations}
    \newtheorem*{convention*}{Convention}
\newtheorem*{acknowledgements*}{Acknowledgements}
\newtheorem*{overview*}{Overview}

\let\oldtocsection=\tocsection

\let\oldtocsubsection=\tocsubsection

\let\oldtocsubsubsection=\tocsubsubsection

\renewcommand{\tocsection}[2]{\hspace{0em}\oldtocsection{#1}{#2}}
\renewcommand{\tocsubsection}[2]{\hspace{1em}\oldtocsubsection{#1}{#2}}
\renewcommand{\tocsubsubsection}[2]{\hspace{2em}\oldtocsubsubsection{#1}{#2}}

\newcommand{\Th}{\mathrm{Th}}

\title{Cobordism of nested manifolds}
\author[A. Send\'{o}n Blanco]{Alba Send\'{o}n Blanco}
\address{Department of Mathematics, Vrije Universiteit Amsterdam}
\email{a.sendon.blanco@vu.nl}

\keywords{Cobordism groups, nested manifolds, links, normal structures, Pontryagin-Thom construction, (generalized) Whitehead products, Hilton-Milnor splitting.}
\makeatletter
\@namedef{subjclassname@2020}{\textup{2020} Mathematics Subject Classification}
\makeatother
\subjclass[2020]{Primary  57R19, 57R90; Secondary 57R15, 57K45, 55Q15}

\begin{document}

\begin{abstract}
We study cobordisms of nested manifolds, which are manifolds together with embedded submanifolds, which can themselves have embedded submanifolds, etc. We identify a nested analog of the Pontryagin-Thom construction. Moreover, when the highest-dimensional manifold has a normal bundle with a framed direction, we find spaces homotopy equivalent to the nested Pontryagin-Thom spaces that relate nested manifolds up to cobordism with links up to cobordism. This gives rise to nested cobordism invariants coming from previously studied cobordism invariants of links. In addition, we provide an alternative proof of a result by Wall about the splitting of the stable nested cobordism groups.
\end{abstract}

\maketitle

\setlength{\parskip}{0px}
\tableofcontents

\setlength{\parskip}{5px}

\section{Introduction}
We say that two closed smooth manifolds are cobordant if there exists a compact manifold bounding their disjoint union. The set of cobordism classes of closed manifolds of dimension $k$ forms a group $\Omega_k$ under disjoint union, and the cartesian product turns the direct sum of all these groups into a graded ring $\Omega_\ast$. The notion of cobordism was first introduced by Poincar\'{e} and developed in more detail by Pontryagin \cite{pontrjagin2007smooth} and Thom \cite{Thom}. In particular, Thom was able to relate the set $\mathrm{Cob}_k(M)$ of $k$-dimensional submanifolds of a fixed closed manifold $M$ up to cobordisms inside $M\times[0,1]$ with a homotopy-theoretic computation, and then calculated the cobordism groups $\Omega_k$ of $k$-manifolds and the cobordism ring $\Omega_\ast$.\par
Cobordism remains a central topic in algebraic topology. For example, Galatius, Madsen, Tillmann and Weiss determined the homotopy type of the cobordism category \cite{galatius2009homotopy}, which was applied to prove a higher-dimensional version of the Madsen-Weiss theorem by Galatius and Randal-Williams \cite{stablemoduli}. The cobordism category, which has manifolds as objects and cobordisms as morphisms, is the domain of functorial Topological Quantum Field Theories (TQFTs), which constitute an active area of research in topology and mathematical physics (see e.g. \cite{freedhopkins}).\par
The central objects of study in this work are nested manifolds, which are manifolds together with embedded submanifolds, which can themselves have embedded submanifolds, etc. Examples of nested manifolds include knots, configurations of points in manifolds and fixed points of manifolds under some group actions. Nested manifolds have a natural notion of cobordism, illustrated in \cref{fig1}.\par
The topic of nested manifolds is not new in the field. Wall \cite{Wallbook,Wallpairs} and Stong \cite{strong1971cobordism} studied cobordisms of nested manifolds under the name of cobordisms of pairs. Hoekzema \cite{Renee} and Ayala \cite[Section 11.5]{ayala} determined the homotopy type of spaces of nested manifolds and related those to the classifying spaces of nested cobordism categories. Calle, Hoekzema, Murray, Pacheco-Tallaj, Rovi and Sridhar-Shapiro examined the generators and relations of the ``striped cylinder'' cobordism category $\mathrm{Cyl}$ with an eye towards the study of TQFTs with domain nested cobordism categories \cite{cyl}. Nested manifolds have also been investigated in the context of cut-and-paste groups by Komiya \cite{komiya} and Vlierhuis \cite{Rolf}.\par
Like in the classical setting, there are two ways of looking at nested manifolds up to cobordism:
\begin{itemize}
    \item \underline{Unstably}: we consider nested manifolds $K'\subseteq K$ of dimensions $k_2< k_1$ inside a closed $m$-manifold $M$ (maybe together with lifts of their normal bundles to, respectively,  $\theta'\colon B'\to BO(k_1-k_2)$ and $\theta\colon B\to BO(m-k_1)$) up to nested cobordism inside $M\times[0,1]$. This gives rise to the set $\mathrm{NCob}^{(\theta',\theta)}(M)$; see \cref{def:nestedmfld,def:nestedcob} for more precise explanations.
    \item \underline{Stably}: we consider nested manifolds inside a high-dimensional sphere $S^m$ up to cobordism inside $S^m\times[0,1]$ for $m>>0$. By Whitney's theorem, this is equivalent to considering abstract nested manifolds $K'\subseteq K$ of dimensions $k_2< k_1$ (maybe together with a lift of the normal bundle of $K'$ to $\theta'\colon B'\to BO(k_1-k_2)$ and a lift of the stable normal bundle of $K$ to $\theta(n)\colon B(n)\to BO(n)$ for some $n$) up to cobordism. This gives rise to the group $\Omega_{k_1}^{(\theta',\Theta)}$, where $\Theta$ encodes all the possible $\theta(n)$'s; see \cref{snm} for a more accurate description. 
\end{itemize}
Often times, looking at stable objects is easier, while unstable objects are more interesting. As explained below, nested cobordism is one more illustration of this heuristic.\par 
It is clear that a necessary condition for a nested manifold $K'\subseteq K$ to be nullbordant is that both $K$ and $K'$ are nullbordant separately. In the stable range, this is also sufficient. Indeed, Wall showed that the stable nested cobordism groups split as a direct sum of classical stable cobordism groups.
\begin{proposition*}[{ \ref{wall} \cite[Lemma 8.3.5]{Wallbook}}] The stable nested cobordism group of $(\theta',\Theta)$-manifolds splits as a direct sum of stable cobordism groups:
    \[\Omega_{k_1}^{(\theta',\Theta)}\cong\Omega_{k_2}^{\theta'\times\Theta}\oplus\Omega_{k_1}^{\Theta}.\]\end{proposition*}
Our work delves further in the splitting problem by analyzing an analog of the Pontryagin-Thom construction for the unstable cobordism sets $\mathrm{NCob}^{(\theta',\theta)}(M)$ and stable cobordism groups $\Omega_{k_1}^{(\theta',\Theta)}$.
\begin{theorem*}[ \ref{nestedcob}] There is a bijection
\[
    \mathrm{NCob}^{(\theta',\theta)}(M)\cong[M,{\Th(\theta'^\ast\gamma_{k_1-k_2}})_+\wedge \Th(\theta^\ast\gamma_{m-k_1})],\]
which becomes a group isomorphism for $m-k_1>\frac{m+1}{2}$ or for $M=S^m$ and $m-k_1>1$. Moreover, there is a group isomorphism
\[\Omega_{k_1}^{(\theta',\Theta)}\cong\pi_{k_1}({\Th(\theta'^\ast\gamma_{k_1-k_2})}_+\wedge \Th\Theta). \]
\end{theorem*}

Then, Wall's result is an immediate corollary of the following proposition.
\begin{proposition*}[ \ref{split}] There is a cofiber sequence of spectra admitting a retract
    \begin{center}
    \begin{tikzcd}
                \Th\Theta \arrow{r}{i} & \Th(\theta'^\ast\gamma_{k_1-k_2})_+\wedge \Th\Theta \arrow[bend right, dashed]{l}[swap]{s} \arrow{r}{\mathrm{col}} & \Th(\theta'^\ast\gamma_{k_1-k_2})\wedge \Th\Theta.
    \end{tikzcd}
\end{center}
\end{proposition*}
It turns out that unstable nested cobordism sets do not split in general. More specifically, when the highest-dimensional manifold has a normal bundle with one framed direction, we can use this direction to move the lowest-dimensional manifold apart from the highest-dimensional one, hence turning our nested manifold into two disjoint manifolds, i.e. a link; see \cref{fig4} for a picture. Links also have a suitable notion of cobordism (see \cref{fig2} for an illustration), which interacts well with the notion of nested cobordism and our unnesting map. We then can relate nested cobordism sets $\mathrm{NCob}^{(\theta',\theta)}(M)$ and cobordism sets of links $\mathrm{LCob}^{(\theta'\times\theta,\theta)}(M)$.
\begin{theorem*}[ \ref{nestsandknots}] For $\theta$ factoring over $BO(m-k_1-1)$, the unnesting map
\[\Upsilon\colon\mathrm{NCob}^{(\theta',\theta)}(M)\to\mathrm{LCob}^{(\theta'\times\theta,\theta)}(M)\]
is bijective.\end{theorem*}
This provides a set of non-trivial nested cobordism invariants, full in the framed case, coming from invariants of cobordism classes of links $\Delta_\lambda$ due to Wang \cite{geohilton,geohiltonthesis} (see \cref{invariant,wang}).
\begin{corollary*}[{ \ref{wang2}}] Let $K'\subseteq K$ be a nested $(\theta',\theta)$-submanifold of $S^m$, with $\theta$ factoring over $BO(m-k_1-1)$. We have that 
\[[K'\subseteq K]=0\text{ in } \mathrm{NCob}^{(\theta',\theta)}(S^m)\xRightarrow{\quad} [K],[K'],\Delta_{[\iota,\iota']}(\Upsilon[K'\subseteq K])=0.\]
\end{corollary*}
\begin{corollary*}[{ \ref{wang3}}]
Let $K'\subseteq K$ be a framed nested submanifold of $S^m$ where $K$ has codimension larger than $1$. We have that 
\[[K'\subseteq K]=0\text{ in } \mathrm{NCob}^{(\ast,\ast)}(S^m)\xLeftrightarrow{\quad}\Delta_\lambda(\Upsilon[K'\subseteq K])=0\ \forall\lambda\in\Lambda,\]
for $\Lambda$ a system of basic Whitehead products.
\end{corollary*}
When we do not have a framed normal direction on the highest-dimensional manifold, the unnesting map cannot be defined and hence the nested cobordism sets become more misterious. However, we know that unstable nested cobordism sets do not necessarily split in this case as well. \cref{favex3} provides an instance of that, which encourages further investigation.
\begin{overview*} This document is structured as follows. In \cref{subsec2.1}, we review the classical Pontryagin-Thom construction. In \cref{subsec2.2}, we introduce the nested cobordism sets and groups and describe the nested Pontryagin-Thom construction. In \cref{sec5}, we show \cref{split} and give a new proof of the splitting of the stable nested cobordism groups. In \cref{subsec3.1}, we define the cobordism sets of links inside a closed manifold, describe a Pontryagin-Thom construction for these and review cobordism invariants for them due to Wang in the case our links have a framed normal direction.
In \cref{subsec3.2}, we prove \cref{nestsandknots} and examine its consequences. Lastly, in \cref{subsec3.3}, we give an example of the nonsplitting of the unstable nested cobordism sets that is not comparable with the case of links, warranting in this way further study.
\end{overview*}

\begin{acknowledgements*} The author is very grateful to her Ph.D advisors, Renee S. Hoekzema and Thomas O. Rot, for their infinite patience and support. She would also like to thank Michael Jung and Lauran Toussaint for very helpful discussions.
\end{acknowledgements*}

\section{Nested Pontryagin-Thom construction}\label{sec2}
\subsection{The classical Pontryagin-Thom construction}\label{subsec2.1}
Let $M$ be a closed $m$-manifold and $k$ an integer with $0\leq k<m$.
\begin{definition}\label[definition]{def:structure} For an integer $d\geq0$, a $d$-dimensional {\bf structure} will be a Serre fibration $B\to BO(d)$.\par
Recall that, up to homotopy equivalence, we can think of any continuous map as a Serre fibration.
\end{definition}
Structures are often used to encode geometric information about manifolds, such as framings. See \cref{pt,singcob} for further details.\par
Fix an $(m-k)$-dimensional structure $\theta\colon B\to BO(m-k)$.
\begin{definition} A {\bf $\theta$-submanifold} of $M$ is a closed $k$-submanifold $K$ of $M$ together with the homotopy class of a lift of the classifying map of the normal bundle $\nu_K^M$ of $K$ inside $M$ to $B$:
\begin{center}
\begin{tikzcd}
    & B\arrow{d}{\theta}\\
    K \arrow[dashed]{ru} \arrow{r}[swap]{\nu_K^M} & BO(m-k).
\end{tikzcd}
\end{center}
\end{definition}
\begin{definition}\label[definition]{def:cob} A {\bf cobordism} between two $\theta$-submanifolds $K$ and $\widetilde{K}$ of $M$ is a compact submanifold $W$ of $M\times[0,1]$, intersecting $M\times\{0,1\}$ transversely, together with a lift of the classifying map of the normal bundle $\nu_W^{M\times[0,1]}$ to $B$, restricting to $K\sqcup\widetilde{K}$ on the boundary:
\begin{center}
    \begin{tikzcd}
        \partial W = W\cap M\times\{0,1\} = K\sqcup\widetilde{K} \arrow[dashed]{r} \arrow[hook]{d} & B \arrow{d}{\theta} \\
        W \arrow{r}[swap]{\nu_{W}^{M\times[0,1]}} \arrow[dashed]{ru} & BO(m-k).
    \end{tikzcd}
\end{center}
This defines an equivalence relation on the set of $\theta$-submanifolds of $M$. Let us denote the {\bf set of cobordism classes of $\theta$-submanifolds of $M$} by $\mathrm{Cob}^\theta(M)$. Note that the codimension of the submanifolds is implicit in the map $\theta$. When the codimension $m-k$ is big enough ($m-k>\frac{m+1}{2}$), $\mathrm{Cob}^\theta(M)$ becomes an abelian group under disjoint union: we can always find disjoint representatives of the classes and the choice of such representatives does not matter up to cobordism. The same happens for smaller codimension ($m-k>1$) in the special case the background manifold $M=S^m$ is a sphere because we can choose our representatives lying in different hemispheres. See \cite[Section IX.3]{kosinski} for a discussion about the framed case; the arguments follow through in general. 
\end{definition}
The pullback of the tautological bundle $\gamma_{m-k}$ over $BO(m-k)$ along $\theta$ is a vector bundle $\theta^\ast\gamma_{m-k}$ of rank $m-k$ over $B$. Its Thom space will be of interest, since the homotopy classes of maps from $M$ to this space are in bijection with the set of cobordism classes of $\theta$-submanifolds of $M$. This bijection becomes a group isomorphism for $m-k>\frac{m+1}{2}$ or for $M=S^m$ and $m-k>1$. Observe that the set of homotopy classes from $M$ to $\Th(\theta^\ast\gamma_{m-k})$ also becomes a group under these assumptions. Indeed, a cell decomposition on $B$ (which we can always have up to homotopy) gives a cell decomposition of the Thom space $\Th(\theta^\ast\gamma_{m-k})$ with one $0$-cell being the infinity point and then one $(l+m-k)$-cell for each $l$-cell of $B$ (see \cite[Lemma 18.1]{milnorstasheff}). In particular, if $M=S^m$ and $m-k>1$, $\Th(\theta^\ast\gamma_{m-k})$ is simply connected and then \[[S^m, \Th(\theta^\ast\gamma_{m-k})]\cong\pi_m(\Th(\theta^\ast\gamma_{m-k})).\] On the other hand, if $m-k>\frac{m+1}{2}$, the inclusion \[\Th(\theta^\ast\gamma_{m-k})\vee\Th(\theta^\ast\gamma_{m-k})\to\Th(\theta^\ast\gamma_{m-k})\times\Th(\theta^\ast\gamma_{m-k})\] is $(2(m-k)-1)$-connected. Since the dimension of $M$ is $m<2(m-k)-1$, we can lift maps $M\to\Th(\theta^\ast\gamma_{m-k})\times \Th(\theta^\ast\gamma_{m-k})$ to the wedge uniquely up to homotopy. Then, given two maps $f,g\colon M\to \Th(\theta^\ast\gamma_{m-k})$, we can lift $(f,g)\colon M\to \Th(\theta^\ast\gamma_{m-k})\times\Th(\theta^\ast\gamma_{m-k})$ to the wedge, and then compose this lift $M\to \Th(\theta^\ast\gamma_{m-k})\vee \Th(\theta^\ast\gamma_{m-k})$ with the fold map $\Th(\theta^\ast\gamma_{m-k})\vee \Th(\theta^\ast\gamma_{m-k})\to \Th(\theta^\ast\gamma_{m-k})$. This defines a binary operation on $[M,\Th(\theta^\ast\gamma_{m-k})]$ with identity the class of the constant map. Moreover, this is a group; see \cite[page 421]{spanier} for further details.
\begin{theorem}[{ (Pontryagin-Thom isomorphism)}]\label[theorem]{PTconst}
There is a bijection
\[\mathrm{Cob}^\theta(M)\cong [M, \Th(\theta^\ast\gamma_{m-k})],\]
which becomes a group isomorphism for $m-k>\frac{m+1}{2}$ or for $M=S^m$ and $m-k>1$.\end{theorem}
\begin{proof}[Sketch of the proof (for full details, see e.g. {\cite[Section 8]{Wallbook}})]
We define the map \[\mathrm{Cob}^\theta(M)\to [M, \Th(\theta^\ast\gamma_{m-k})]\] in the following way. Take a $\theta$-submanifold $K$ of $M$ representing a cobordism class and define a map $M\to \Th(\theta^\ast\gamma_{m-k})$ by sending a tubular neighbourhood of $K$ to $\Th(\theta^\ast\gamma_{m-k})$ via the lift of the classifying map of the normal bundle of $K$ inside $M$, and the rest of $M$ to the infinity point of $\Th(\theta^\ast\gamma_{m-k})$.\par
The inverse map works as follows. Take a map $M\to\Th(\theta^\ast\gamma_{m-k})$ representing a homotopy class. Let $\gamma_{m-k,n}$ be the tautological bundle over the Grassmannian $\mathrm{Gr}_{m-k}(\mathbb{R}^n)$ of $(m-k)$-planes in $\mathbb{R}^n$. Since $M$ is compact, the composition $M\to \Th(\theta^\ast\gamma_{m-k})\to\Th(\gamma_{m-k})$ factors through $\Th(\gamma_{m-k,n})$, for some $n$. Choose the representative so that $M\to \Th(\gamma_{m-k,n})$ is differentiable and transverse to $\mathrm{Gr}_{m-k}(\mathbb{R}^n)$. Then, the inverse image of the Grassmannian gives us a $k$-submanifold of $M$ whose normal bundle inside $M$ has a lift to $B$.\end{proof}
\begin{example}\label[example]{pt}
Setting $\theta\simeq\ast\colon\ast\to BO(m-k)$, the trivial structure, \cref{PTconst} recovers the classical result by Pontryagin \cite{pontrjagin2007smooth}. In this case, $\mathrm{Cob}^{\ast}(M)=\mathrm{Cob}_k^{fr}(M)$ is the set of cobordism classes of framed $k$-submanifolds of $M$:
\[\mathrm{Cob}_k^{fr}(M)\cong [M, S^{m-k}].\]
\cref{PTconst} gets back the work of Thom \cite[Theorem IV.6]{Thom} for the identity structure $\theta=\mathrm{id}\colon BO(m-k)\to BO(m-k)$. In this case, we are considering the set of cobordism classes of all $k$-submanifolds of $M$, $\mathrm{Cob}^{\mathrm{id}}_k(M)=\mathrm{Cob}_k(M)$:
\begin{align*}
    \mathrm{Cob}_k(M)&\cong [M, \Th(\gamma_{m-k})].
\end{align*}
\end{example}
Now, rather than only looking at submanifolds of a given manifold, we would like to look at abstract manifolds. We can do this by considering submanifolds of increasingly big spheres, but then the rank of our normal bundles is not fixed. We hence need to consider stable structures.
\begin{definition}\label[definition]{def:stablestructure} A {\bf stable structure} $\Theta$ is a commutative diagram of spaces:
\begin{center}
    \begin{tikzcd}
        \ldots \arrow{r} & B(n-1) \arrow{d}{\theta({n-1})} \arrow{r} & B(n) \arrow{d}{\theta(n)} \arrow{r} & B(n+1) \arrow{r} \arrow{d}{\theta({n+1})} & \ldots\\
        \ldots \arrow{r} & BO(n-1) \arrow{r} & BO(n) \arrow{r} & BO(n+1) \arrow{r} & \ldots,
    \end{tikzcd}
\end{center}
where $\theta(n)$ is a Serre fibration for every $n\geq0$.
\end{definition}
We use capital letters to distinguish stable structures from the unstable ones.\par
Fix a stable structure $\Theta$.
\begin{definition}
A {\bf $\Theta$-manifold} of dimension $k$ is a closed $k$-manifold $K$ together with the homotopy class of a lift of the classifying map of the normal bundle of some embedding $K\hookrightarrow S^{n+k}$ to $B(n)$.\par
We will identify an embedding $K\hookrightarrow S^{n+k}$ with all the embeddings $K\hookrightarrow S^{n'+k}$ factoring over the standard inclusion $S^{n+k}\hookrightarrow S^{n'+k}$ for $n'>n$: $K\hookrightarrow S^{n+k}\hookrightarrow S^{n'+k}$.\par
Likewise, we will identify a lift of the classifying map of the normal bundle of $K\hookrightarrow S^{n+k}$ to $B(n)$ with the composition of this lift with $B(n)\to B(n+1)\to\ldots$.\par
Two $\Theta$-manifolds of dimension $k$ will be cobordant if they are cobordant as $\theta(n)$-submanifolds of $S^{n+k}$ for some $n$.\par
The {\bf stable cobordism group} $\Omega_{k}^{\Theta}$ of $\Theta$-manifolds of dimension $k$ is:
\begin{align*}\Omega_{k}^{\Theta}=\mathrm{colim}_{n\to\infty}\mathrm{Cob}^{\theta(n)}(S^{k+n}),\end{align*}
the colimit of the maps $\mathrm{Cob}^{\theta(n)}(S^{k+n})\to\mathrm{Cob}^{\theta(n+1)}(S^{k+n+1})$ induced by the standard sphere inclusions.
\end{definition}
By putting all the Thom spaces $\{\Th(\theta(n)^\ast\gamma_{n})\}_{n\geq1}$ together, one gets the Thom spectrum $\Th\Theta$, and then the stable cobordism groups are isomorphic to the homotopy groups of this spectrum. We refer the reader to \cite[Chapter III.2]{spectra} for detailed definitions of spectra and stable homotopy groups.
\begin{theorem}[{ \cite{lashof}}]\label[theorem]{stablePTcons} There is a group isomorphism:
\[\Omega_k^\Theta\cong \pi_k(\Th\Theta).\]
\end{theorem}
\begin{example} Let us stabilize \cref{pt}. Setting $\Theta\simeq\ast$ with $\theta(n)\simeq\ast$, we get that the cobordism groups of stably normally framed manifolds are isomorphic to the stable homotopy groups of spheres: $\Omega_k^{fr}\cong\pi_k \mathbb{S}$. Placing $\Theta=\mathrm{id}$ with $\theta(n)=\mathrm{id}$, we get that the cobordism groups of manifolds are isomorphic to the homotopy groups of the Thom spectrum $MO$: $\Omega_k\cong\pi_k MO$. Thom \cite[Theorem IV.12]{Thom} was able to compute these cobordism groups and the cobordism ring $\Omega_\ast$.\end{example}

\begin{example}\label[example]{singcob}
The following structure will be important for us. Let $X$ be a topological space and consider
$BO(m-k)\times X\to BO(m-k)$ the projection onto the first factor. A submanifold of $M$ with a lift to this map is a closed submanifold of $M$ together with the homotopy class of a map to $X$. Two such submanifolds of $M$ are cobordant if there is a compact submanifold of $M\times [0,1]$, together with a map to $X$, restricting to our initial submanifolds and maps on the boundary. This is also called singular cobordism in $X$. Note that in this instance, the Thom space is $X_+\wedge\Th(\gamma_{m-k})$.\par
More generally, given a structure $\theta\colon B\to BO(m-k)$ and a topological space $X$, one gets a new structure
\[\theta_X\colon B\times X\to BO(m-k)\]
by composing the projection onto the first factor with $\theta$. A $\theta_X$-submanifold of $M$ is a $\theta$-submanifold of $M$ together with the homotopy class of a map from the submanifold to $X$. Similarly as before, 
\begin{equation}\label{eq9}
\Th(\theta_X^\ast\gamma_{m-k})\cong X_+\wedge\Th(\theta^\ast\gamma_{m-k}).
\end{equation}
For $\Theta=\Theta(X)$ with $\theta(n)\colon BO(n)\times X\to BO(n)$, \cref{stablePTcons} recovers part of the fact that singular cobordism in $X$ (usually denoted by $\Omega_n(X)=\Omega_n^{\Theta(X)}$) is a homology theory \cite{atiyahbord,connerfloyd}: $\Omega_n^{\Theta(X)}\cong \pi_n(X_+\wedge MO)$.\end{example}

\subsection{Nested manifolds and the nested Pontryagin-Thom construction}\label{subsec2.2}
We will first define once nested submanifolds of an $m$-dimensional closed manifold $M$. Fix integers $k_1,k_2$ with $0\leq k_2<k_1<m$ and consider structures $\theta\colon B\to BO(d)$ and $\theta'\colon B'\to BO(d')$ as in \cref{def:structure} for $d=m-k_1$ and $d'=k_1-k_2$.
\begin{definition}\label[definition]{def:nestedmfld}
A {\bf (once) nested $(\theta',\theta)$-submanifold} of $M$ is a $\theta$-submanifold $K$ of $M$ together with a $\theta'$-submanifold $K'$ of $K$. We will denote this by $K'\subseteq K\subseteq M$. See \cref{fig1} for an example of a nested submanifold of $S^2$.
\end{definition}

\begin{definition}\label[definition]{def:nestedcob}
A {\bf nested $(\theta',\theta)$-cobordism} between two nested  $(\theta',\theta)$-submanifolds $K'\subseteq K$, $\widetilde{K}'\subseteq\widetilde{K}$ of $M$ is a cobordism $W$ between $K$ and $\widetilde{K}$ as in \cref{def:cob}, together with a compact submanifold $W'\subseteq W$ intersecting $\partial W$ transversely and a lift of the classifying map of the normal bundle $\nu_{W'}^W$ to $\theta'$, restricting to $K'\sqcup\widetilde{K}'$ on the boundary:
\begin{center}
    \begin{tikzcd}
        \partial W' = W'\cap M\times\{0,1\} = K'\sqcup\widetilde{K}' \arrow[dashed]{r} \arrow[hook]{d} & B' \arrow{d}{\theta'} \\
        W' \arrow{r}[swap]{\nu_{W'}^{W}} \arrow[dashed]{ru} & BO(d').
    \end{tikzcd}
\end{center}
See \cref{fig1} for an example of a nested cobordism inside $S^2\times[0,1]$.
\end{definition}
\begin{figure}
    \centering
    \includegraphics[scale=0.55]{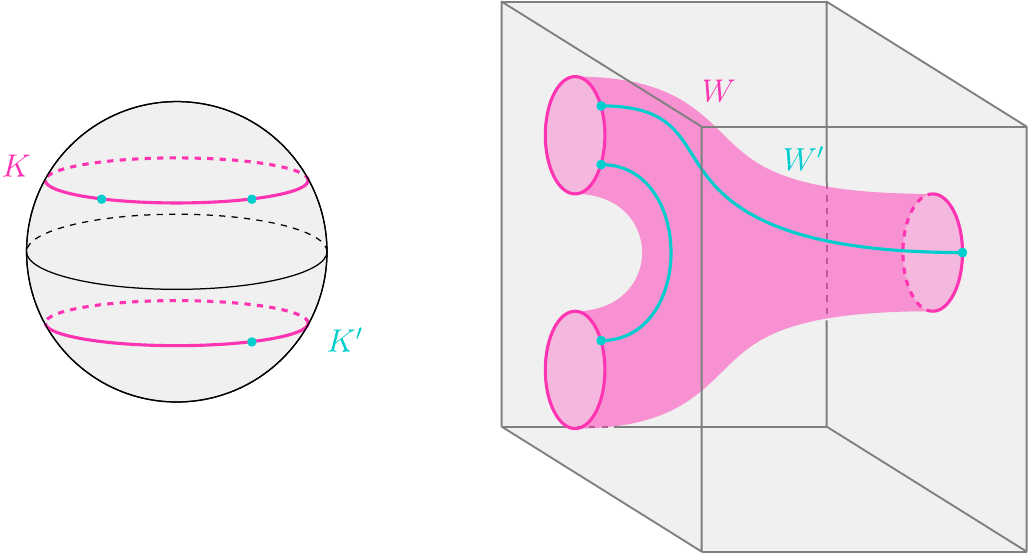}
    \caption{On the left, a nested submanifold $K'\subseteq K$ of $S^2$; on the right, a nested cobordism $W'\subseteq W$ inside $S^2\times[0,1]$.}
    \label{fig1}
\end{figure}
Nested $(\theta',\theta)$-cobordism is an equivalence relation. Let us denote the {\bf set of $(\theta',\theta)$-cobordism classes of nested  $(\theta',\theta)$-submanifolds of $M$} by $\mathrm{NCob}^{(\theta',\theta)}(M)$. Notice that the codimensions of the submanifolds are implicit in the structures $\theta$, $\theta'$. If we have big enough codimension $d>\frac{m+1}{2}$, $\mathrm{NCob}^{(\theta',\theta)}(M)$ is an abelian group under disjoint union. The same happens if $d>1$ and $M= S^m$.\par
Next, we introduce the definition of nested manifold without a background manifold $M$. Note that in this case, the rank of the normal bundle of the highest-dimensional manifold changes, but the rank of the normal bundle of the smallest-dimensional manifold inside the highest-dimensional manifold is still fixed. Then, we consider one stable structure $\Theta$ as in \cref{def:stablestructure} for the highest-dimensional manifold.
\begin{definition}\label[definition]{snm} A {\bf (once) nested $(\theta',\Theta)$-manifold} of dimensions $k_2<k_1$ is a $\Theta$-manifold $K$ of dimension $k_1$, together with a $\theta'$-submanifold $K'\subseteq K$.\par
The {\bf stable nested cobordism group $\Omega_{k_1}^{(\theta',\Theta)}$ of $(\theta',\Theta)$-manifolds of dimensions $k_2<k_1$} is:
\begin{align*}\Omega_{k_1}^{(\theta',\Theta)}=\mathrm{colim}_{n\to\infty}\mathrm{NCob}^{(\theta',\theta(n))}(S^{k_1+n}),\end{align*}
the colimit of the maps $\mathrm{NCob}^{(\theta',\theta(n))}(S^{k_1+n})\to\mathrm{NCob}^{(\theta',\theta(n+1))}(S^{k_1+n+1})$ induced by the standard sphere inclusions.\end{definition}
Cobordisms between submanifolds of a given manifold $M$ were already studied by Pontryagin and Thom. However, these cobordisms were considered within $M\times [0,1]$, that is, with the highest-dimensional cobordism being cylindrical. We allow the highest-dimensional manifold to change, and we will also consider the option of having further submanifolds inside the submanifolds.\par
The unstable and the stable nested Pontryagin-Thom constructions are given by the following theorem.
\begin{theorem}\label[theorem]{nestedcob} There is a bijection
\[
    \mathrm{NCob}^{(\theta',\theta)}(M)\cong[M,{\Th(\theta'^\ast\gamma_{d'})}_+\wedge \Th(\theta^\ast\gamma_{d})],\]
where $d=m-k_1$ and $d'=k_1-k_2$, which becomes a group isomorphism for $m-k_1>\frac{m+1}{2}$ or for $M=S^m$ and $m-k_1>1$. Moreover, there is a group isomorphism
\[\Omega_{k_1}^{(\theta',\Theta)}\cong\pi_{k_1}({\Th(\theta'^\ast\gamma_{d'})}_+\wedge \Th\Theta). \]
\end{theorem}
\begin{proof} Let $K$ be a $\theta$-submanifold of $M$. \cref{PTconst} tells us that maps $K\to \Th(\theta'^\ast\gamma_{d'})$ up to homotopy give us $\theta'$-submanifolds of $K$ up to cobordism. Hence, one could look at nested cobordism sets/groups as singular cobordism sets/groups in $\Th(\theta'^\ast\gamma_{d'})$ as in \cref{singcob}. That is,
\[
    \mathrm{NCob}^{(\theta',\theta)}(M)\cong\mathrm{Cob}^{\theta_{\Th(\theta'^\ast\gamma_{d'})}}(M).\]
Then, by \cref{PTconst}, the nested cobordism sets/groups will be isomorphic to homotopy classes of maps to some Thom space,
\[\mathrm{Cob}^{\theta_{\Th(\theta'^\ast\gamma_{d'})}}(M)\cong[M, \Th(\theta_{\Th(\theta'^\ast\gamma_{d'})}^\ast\gamma_{d})].\]
Lastly, we have seen in \eqref{eq9} that this Thom space has a particular form:
 \[[M, \Th(\theta_{\Th(\theta'^\ast\gamma_{d'})}^\ast\gamma_{d})]\cong[M,{\Th(\theta'^\ast\gamma_{d'})}_+\wedge \Th(\theta^\ast\gamma_{d})].\]
Composing the three isomorphisms above yields the first of the claims.\par
Similarly, in the stable case:
\[\Omega_{k_1}^{(\theta',\Theta)}\cong\Omega_{k_1}^{\Theta_{\Th(\theta'^\ast\gamma_{d'})}}\stackrel{\ref{stablePTcons}}{\cong}\pi_{k_1}({\Th(\theta'^\ast\gamma_{d'})}_+\wedge \Th\Theta).\qedhere\] 
\end{proof}
Loosely speaking, the cobordism class of a nested $(\theta',\theta)$-submanifold $K'\subseteq K$ of $M$ is sent to the homotopy class of a map $M\to \Th(\theta'^\ast\gamma_{d'})_+\wedge \Th(\theta^\ast\gamma_{d})$ in such a way that the inverse image of $\Th(\theta'^\ast\gamma_{d'})\times B$ is $K$ and the inverse image of $B'\times B$ is $K'$.\par 
This was already noted by Stong in \cite{cobofmaps}. There, he develops a notion of cobordism of maps: two continuous maps $f\colon K\to L$, $\widetilde{f}\colon \widetilde{K}\to\widetilde{L}$ between closed manifolds are cobordant if there is a cobordism $W$ between $K$ and $\widetilde{K}$, a cobordism $Y$ between $L$ and $\widetilde{L}$, and a map $F\colon W\to Y$ restricting to $f\sqcup\widetilde{f}$ on the boundary. Stong proves that the cobordism group $\Omega(m,n)$ of maps from an $m$-dimensional manifold to an $n$-dimensional manifold is isomorphic to a singular cobordism group as in \cref{singcob}: $\Omega(m,n)\cong\Omega_n^{\Theta({\mathbf{MO}})}$, where $\mathbf{MO}= \Omega^\infty\Sigma^{n-m}MO$ is the infinite loop space coming from a shift of the Thom spectrum. He furthermore adds that classes containing embeddings (which can be thought of classes represented by nested manifolds) come from $\Omega_n^{\Theta({\Th(\gamma_{n-m})})}\to\Omega_n^{\Theta({\mathbf{MO}})}$. The stable part of \cref{nestedcob} appears also in \cite[Lemma 8.4.9]{Wallbook}.\par
\begin{remark} If we wish to nest our manifolds even more, we just have to iterate this process. For example, if we set integers $0\leq k_3<k_2<k_1<m$ and structures $\theta\colon B\to BO(d)$, $\theta'\colon B'\to BO(d')$ and  $\theta''\colon B''\to BO(d'')$ as in \cref{def:structure} for $d=m-k_1$, $d'=k_1-k_2$, $d''=k_2-k_3$, a twice nested $(\theta'',\theta',\theta)$-submanifold of a closed $m$-manifold $M$ would be a $\theta$-submanifold $K$ of $M$ together with a $\theta'$-submanifold $K'$ of $K$ and a $\theta''$-submanifold $K''$ of $K'$. For $\Theta$ a stable structure as in \cref{def:stablestructure}, a twice nested $(\theta'',\theta',\Theta)$-manifold of dimensions $0\leq k_3<k_2<k_1$ would be a $\Theta$-manifold of dimension $k_1$ together with a $\theta'$-submanifold, which itself has a $\theta''$-submanifold. The twice nested cobordism sets/groups would then be:
\begin{align*}
    \mathrm{NCob}^{(\theta'',\theta',\theta)}(M)&\cong\mathrm{Cob}^{\theta_{\left(\Th(\theta''^\ast\gamma_{d''})_+\wedge\Th(\theta'^\ast\gamma_{d'})\right)}}(M)\\
    &\stackrel{\ref{PTconst}}{\cong}[M,\left(\Th(\theta''^\ast\gamma_{d''})_+\wedge\Th(\theta'^\ast\gamma_{d'})\right)_+\wedge \Th(\theta^\ast\gamma_{d})],\\
\Omega_{k_1}^{(\theta'',\theta',\Theta)}&\cong\Omega_{k_1}^{\Theta_{\left(\Th(\theta''^\ast\gamma_{d''})_+\wedge\Th(\theta'^\ast\gamma_{d'})\right)}}\\
&\stackrel{\ref{stablePTcons}}{\cong}\pi_{k_1}(\left(\Th(\theta''^\ast\gamma_{d''})_+\wedge\Th(\theta'^\ast\gamma_{d'})\right)_+\wedge\Th\Theta). 
\end{align*}
\end{remark}

\subsection{Splitting of the stable nested cobordism groups}\label{sec5}
In this section, we prove that the stable nested cobordism groups split as a direct sum of cobordism groups. Consider integers $k_1,k_2$ with $0\leq k_2<k_1$, a structure $\theta'\colon B'\to BO(d')$ as in \cref{def:structure} for $d'=k_1-k_2$ and a stable structure $\Theta$ as in \cref{def:stablestructure}.\par
First, let us see that the nested Thom spectra $\Th(\theta'^\ast\gamma_{d'})_+\wedge \Th\Theta$ of \cref{nestedcob} are involved in a convenient cofiber sequence.
\begin{proposition}\label[proposition]{split} There is a cofiber sequence of spectra admitting a retract:
    \begin{equation}\label{spliteq}
    \begin{tikzcd}
                \Th\Theta \arrow{r}{i} & \Th(\theta'^\ast\gamma_{d'})_+\wedge \Th\Theta \arrow[bend right, dashed]{l}[swap]{s} \arrow{r}{\mathrm{col}} & \Th(\theta'^\ast\gamma_{d'})\wedge \Th\Theta.
    \end{tikzcd}
\end{equation}
\end{proposition}
\begin{proof} The result follows from smashing with $\Th\Theta$ the following cofiber sequence of pointed spaces admitting a retract:
    \begin{center}
    \begin{tikzcd}
                S^0 \arrow{r}{i'} & \Th(\theta'^\ast\gamma_{d'})_+\arrow[bend right, dashed]{l}[swap]{s'} \arrow{r}{\mathrm{col}'} & \Th(\theta'^\ast\gamma_{d'}),
    \end{tikzcd}
\end{center}
where $i'$ takes the non-basepoint to the infinity point of the Thom space, $s'$ sends the whole Thom space to the non-basepoint and $\mathrm{col}'$ is the quotient map.\end{proof}
Observe now that the spectrum on the right hand side of the cofiber sequence \eqref{spliteq} is also a Thom spectrum for some structure as a consequence of the following lemma by Atiyah \cite{atiyah}.
\begin{lemma}[{ \cite[Lemma 2.3]{atiyah}}]\label[lemma]{atiyah}
Let $A$ and $B$ be finite $CW$-complexes, $\alpha$ and $\beta$ vector bundles over $A$ and $B$ respectively. Then, the Thom space of the vector bundle $\alpha\times\beta$ over $A\times B$ is naturally homeomorphic to $\Th\alpha\wedge\Th\beta$.\end{lemma}
The splitting of the stable nested cobordism groups now follows from \cref{split,atiyah}.
\begin{proposition}[ {\cite[Lemma 8.3.5]{Wallbook}}]\label[proposition]{wall} The stable nested cobordism group of $(\theta',\Theta)$-manifolds splits as a direct sum of stable cobordism groups:
    \[\Omega_{k_1}^{(\theta',\Theta)}\cong\Omega_{k_2}^{\theta'\times\Theta}\oplus\Omega_{k_1}^{\Theta}.\]\end{proposition}
\begin{proof} Recall that $d'=k_1-k_2$. First, by \cref{nestedcob}, we know that the stable nested cobordism groups are stable homotopy groups of some spectrum that admits a splitting by \cref{split}:
\[\Omega_{k_1}^{(\theta',\Theta)}\cong\pi_{k_1}(\Th(\theta'^\ast\gamma_{d'})_+\wedge \Th\Theta)\cong\pi_{k_1}(\Th(\theta'^\ast\gamma_{d'})\wedge \Th\Theta)\oplus\pi_{k_1}(\Th\Theta).\]
Next, notice that $\Th(\theta'^\ast\gamma_{d'})\wedge \Th\Theta\cong\mathrm{sh}^{k_1-k_2}\Th(\theta'\times\Theta)$, where $\mathrm{sh}^k$ denotes the $k$-fold shift operator and the stable structure $\theta'\times\Theta$ is constructed with the maps
\[B'\times B(n)\stackrel{\theta'\times\theta(n)}{\longrightarrow}BO(k_1-k_2)\times BO(n)\to BO(k_1-k_2+n).\]
Indeed, 
\begin{align*}
(\Th(\theta'^\ast\gamma_{d'})\wedge \Th\Theta)_n&=\Th(\theta'^\ast\gamma_{d'})\wedge \Th(\theta(n)^\ast\gamma_n)\\
&\stackrel{\ref{atiyah}}{\cong} \Th(\theta'^\ast\gamma_{d'}\times \theta(n)^\ast\gamma_n)=\Th(\theta'\times\Theta)_{k_1-k_2+n},
\end{align*}
and the structure maps also agree. In consequence, \[\pi_{k_1}(\Th(\theta'^\ast\gamma_{d'})\wedge \Th\Theta)\cong\pi_{k_2}(\Th(\theta'\times\Theta)).\]

Lastly, \cref{PTconst} tells us that the stable homotopy groups obtained are classic stable cobordism groups:
\[\pi_{k_2}(\Th(\theta'\times\Theta))\oplus\pi_{k_1}(\Th\Theta)\cong\Omega_{k_2}^{\theta'\times\Theta}\oplus\Omega_{k_1}^{\Theta}.\qedhere\]
\end{proof}
Wall provides a geometric proof of this result in \cite{Wallpairs} and \cite[Chapter 8]{Wallbook}. In \cite{strong1971cobordism}, Stong generalized this result allowing additional structure on the lowest-dimensional manifold.

\section{Cobordism of nested manifolds versus cobordism of links}\label{sec3}

In this section, we will compare nested manifolds up to cobordism with links up to cobordism in the case that the highest-dimensional manifold of our nested manifold has a normal bundle with a framed direction.

\subsection{Cobordism of links}\label{subsec3.1}

Let us first define the cobordism sets of links inside a closed $m$-manifold $M$. Consider integers $k_1,k_2$ with $0\leq k_2, k_1<m$ and structures $\theta\colon B\to BO(d)$, $\theta'\colon B'\to BO(d')$ as in \cref{def:structure} for $d=m-k_1$, $d'=m-k_2$. 

\begin{definition}
A {\bf $(\theta',\theta)$-link inside $M$} is a pair of submanifolds of $M$, $K$ and $K'$, such that $K$ is a $\theta$-submanifold, $K'$ is a $\theta'$-submanifold and the two are disjoint. See \cref{fig2} for an example of a link inside $S^2$. \end{definition}
\begin{definition}
Two closed $(\theta',\theta)$-links $K\sqcup K'$ and $\widetilde{K}\sqcup \widetilde{K}'$ inside $M$ are {\bf cobordant} if there exist a $\theta$-cobordism $W\subseteq M\times[0,1]$ from $K$ to $\widetilde{K}$ and a $\theta'$-cobordism $W'\subseteq M\times[0,1]$ from $K'$ to $\widetilde{K}'$ as in \cref{def:cob} in such a way that $W\cap W'=\varnothing$. We will call $W\sqcup W'$ a $(\theta',\theta)$-cobordism between $K\sqcup K'$ and $\widetilde{K}\sqcup \widetilde{K}'$. See \cref{fig2} for an example of a cobordism of links inside $S^2\times[0,1]$.
\end{definition}
\begin{figure}
    \centering
    \includegraphics[scale=0.55]{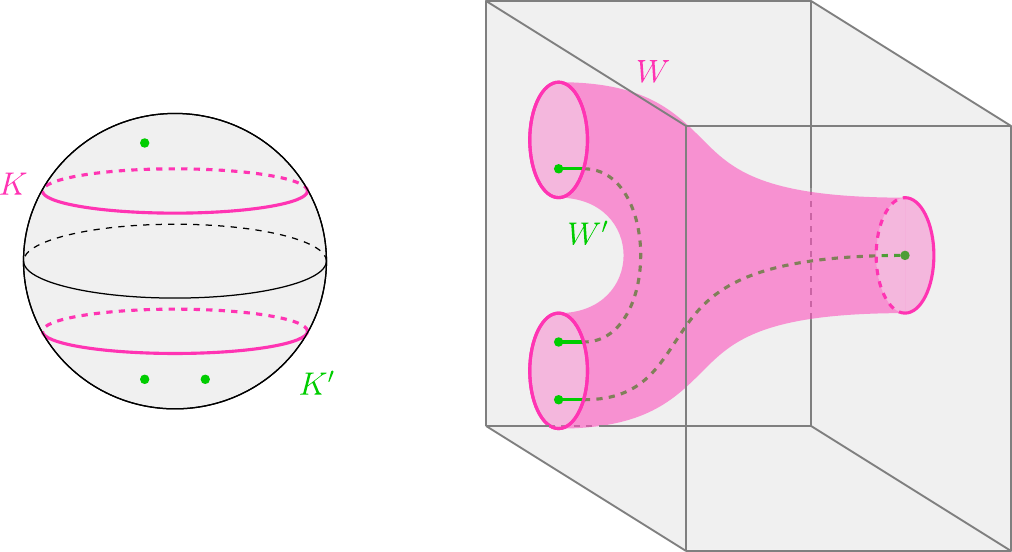}
    \caption{On the left, a link $K\sqcup K'$ inside $S^2$; on the right, a cobordism of links $W\sqcup W'$ inside $S^2\times[0,1]$.}
    \label{fig2}
\end{figure}
Again, $(\theta',\theta)$-cobordism of links is an equivalence relation. Let us denote the {\bf set of cobordism classes of $(\theta',\theta)$-links inside $M$} by $\mathrm{LCob}^{(\theta',\theta)}(M)$. As before, we get a group when we have big enough codimension $\min(d,d')>\frac{m+1}{2}$ or when $\min(d,d')>1$ and $M=S^m$.\par
One can perform a Pontryagin-Thom construction for links inside $M$. Our target spaces are wedges of Thom spaces. The following proposition is proved in \cite{haefligersteer} for the framed case and is used in \cite{geohilton,geohiltonthesis}.
\begin{proposition}\label[proposition]{linkPT} There is a bijection
\[\mathrm{LCob}^{(\theta',\theta)}(M)\cong\left[M, \Th(\theta'^\ast\gamma_{d'})\vee\Th(\theta^\ast\gamma_{d})\right],\]
which becomes a group isomorphism for $d,d'>\frac{m+1}{2}$ or for $M=S^m$ and $d,d'>1$.\end{proposition}
\begin{proof} The proof is very similar to that of \cref{PTconst}.\par
We define the map $\mathrm{LCob}^{(\theta',\theta)}(M)\to\left[M, \Th(\theta'^\ast\gamma_{d'})\vee\Th(\theta^\ast\gamma_{d})\right]$ in the following way. Take a $(\theta',\theta)$-link $K\sqcup K'$ inside $M$ representing a cobordism class and define a map $M\to \Th(\theta'^\ast\gamma_{d'})\vee\Th(\theta^\ast\gamma_{d})$ by sending a tubular neighbourhood of $K$ to $\Th(\theta^\ast\gamma_{d})$ and a tubular neighbourhood of $K'$ (disjoint to the one of $K$) to $\Th(\theta'^\ast\gamma_{d'})$ via the lifts of the classifying maps of their normal bundles inside $M$, and the rest of $M$ to the infinity point of the wedge. Applying the same technique to a cobordism of links, we get a homotopy between the maps we have constructed for each of the links. Hence, the map is well defined. Additionally, it behaves well with respect to disjoint union.\par
Let us now define the inverse. Take a map $M\to\Th(\theta'^\ast\gamma_{d'})\vee\Th(\theta^\ast\gamma_{d})$ representing a homotopy class. The composition $M\to \Th(\theta'^\ast\gamma_{d'})\vee\Th(\theta^\ast\gamma_{d})\to\Th(\gamma_{d'})\vee\Th(\gamma_{d})$ factors through $\Th(\gamma_{d',n})\vee\Th(\gamma_{d,n})$ for some $n$ as $M$ is compact. Choose the representative so that $M\to \Th(\gamma_{d',n})\vee\Th(\gamma_{d,n})$ is differentiable and transverse to $\mathrm{Gr}_{d'}(\mathbb{R}^n)$ and $\mathrm{Gr}_{d}(\mathbb{R}^n)$. Then, the inverse image of the Grassmannians gives us a $(\theta',\theta)$-link inside $M$. Applying the same technique to a homotopy between two maps, we get a $(\theta',\theta)$-cobordism between the resulting links. Hence, the inverse is well defined.\par
By construction, the composition
\[\mathrm{LCob}^{(\theta',\theta)}(M)\to\left[M, \Th(\theta'^\ast\gamma_{d'})\vee\Th(\theta^\ast\gamma_{d})\right]\to\mathrm{LCob}^{(\theta',\theta)}(M)\]
is the identity. The composition
\[\left[M, \Th(\theta'^\ast\gamma_{d'})\vee\Th(\theta^\ast\gamma_{d})\right]\to\mathrm{LCob}^{(\theta',\theta)}(M)\to\left[M, \Th(\theta'^\ast\gamma_{d'})\vee\Th(\theta^\ast\gamma_{d})\right]\]
is also the identity. Indeed, take a class $[f]\in\left[M, \Th(\theta'^\ast\gamma_{d'})\vee\Th(\theta^\ast\gamma_{d})\right]$ with a convenient representative $f$ as above. Applying the second map we have defined gives us a $(\theta',\theta)$-link $K\sqcup K'$ inside $M$. After applying the first map we have defined to this link, we obtain a map $g\colon M\to \Th(\theta'^\ast\gamma_{d'})\vee\Th(\theta^\ast\gamma_{d})$ with $f\big|_{K\sqcup K'}\simeq g\big|_{K\sqcup K'}$. Since a neighbourhood $N$ of $K\sqcup K'$ retracts onto the link, we also get that $f\big|_N\simeq g\big|_N$. Finally, the complement of $N$ is mapped to the contractible space $\left(\Th(\theta'^\ast\gamma_{d'})\vee\Th(\theta^\ast\gamma_{d})\right)\setminus(B'\sqcup B)$, so we can conclude that $[f]=[g]$.\end{proof}
\begin{remark}
By definition, a necessary condition for a link $K\sqcup K'$ inside $M$ to be nullbordant is that both submanifolds $K$, $K'$ of $M$ are nullbordant separately. Under \cref{linkPT}, the forgetful maps 
\begin{align*}
    \Delta_\iota\colon\mathrm{LCob}^{(\theta',\theta)}(M)&\to\mathrm{Cob}^{\theta}(M),& \Delta_{\iota'}\colon\mathrm{LCob}^{(\theta',\theta)}(M)&\to\mathrm{Cob}^{\theta'}(M)\\
    [K\sqcup K']&\mapsto[K]& [K\sqcup K']&\mapsto[K']
\end{align*}
correspond to the projections 
\[p\colon \Th(\theta'^\ast\gamma_{d'})\vee\Th(\theta^\ast\gamma_{d})\to\Th(\theta^\ast\gamma_{d}),\hspace{4em} p'\colon\Th(\theta'^\ast\gamma_{d'})\vee\Th(\theta^\ast\gamma_{d})\to \Th(\theta'^\ast\gamma_{d'}).\]
This is, the following diagram commutes:
\begin{equation}\label{square}
    \begin{tikzcd}        \left[M,\Th(\theta^\ast\gamma_{d-1})\vee\Th(\theta'^\ast\gamma_{d'-1})\right] \arrow{r}{p\circ -} \arrow{d}{\ref{linkPT}}[swap]{\cong} &  \left[M,\Th(\theta^\ast\gamma_{d-1})\right] \arrow{d}{\ref{PTconst}}[swap]{\cong} \\  \mathrm{LCob}^{(\theta',\theta)}(M)  \arrow{r}{\Delta_{\iota}}  &  \mathrm{Cob}^{\theta}(M),
    \end{tikzcd}
\end{equation}
and analogously for $\Delta_{\iota'}$ and $p'$.\par
In the next subsection, we will see that this necessary condition is not always sufficient.\end{remark}

\subsection{The Hilton-Milnor splitting and Wang's invariants for cobordism of links}
Fix integers $m,k_1,k_2$ with $0\leq k_2,k_1<m$ and consider structures $\theta\colon B\to BO(d)$ and $\theta'\colon B'\to BO(d')$ as in \cref{def:structure} for $d=m-k_1$ and $d'=m-k_2$.\par
In the previous subsection, we saw that the question 
\begin{center}``When is it sufficient for a $(\theta',\theta)$-link inside $S^m$ to be nullbordant that the $\theta$-link $K$\\
inside $S^m$ and the $\theta'$-link $K'$ inside $S^m$ are nullbordant separately?"\end{center}
is equivalent to the problem
\begin{center}
``When does the set of homotopy classes $[S^m,\Th(\theta'^\ast\gamma_{d'})\vee\Th(\theta^\ast\gamma_{d})]$ split as\\
the product of $[S^m,\Th(\theta'^\ast\gamma_{d'})]$ and $[S^m,\Th(\theta^\ast\gamma_{d})]$?''.
\end{center}
In the case our Thom spaces are suspensions, Hilton and Milnor \cite{hilton,milnor} gave a response to the latter, which shows that the cobordism class of a link is not necessarily determined by the cobordism classes of their components.\par
Let us hence restrict to the background manifold $S^m$ and the special case when our normal structures $\theta$, $\theta'$ factor over $BO(d-1)$, $BO(d'-1)$. That is, there is a commutative diagram
\begin{center}
    \begin{tikzcd}
        B \arrow{r}{\widetilde{\theta}} \arrow[bend right]{rr}{\theta} & BO(d-1) \arrow{r}{i} & BO(d),
    \end{tikzcd}
\end{center}
for $\theta$ and an analogous one for $\theta'$. In this situation, the submanifolds have a normal bundle with one framed direction and the Thom spaces become suspensions:
\[
\Th(\theta^\ast\gamma_{d})\cong\Th(\widetilde{\theta}^\ast(\varepsilon\oplus\gamma_{d-1}))\cong\Th(\varepsilon\oplus\widetilde{\theta}^\ast\gamma_{d-1})\cong\Sigma\Th(\widetilde{\theta}^\ast\gamma_{d-1}),\]
and analogously $\Th(\theta'^\ast\gamma_{d'})\cong\Sigma\Th(\widetilde{\theta}'^\ast\gamma_{d'-1})$. In particular, the homotopy groups of the Pontryagin-Thom spaces for links have a splitting due to Hilton and Milnor \cite{hilton,milnor}. Let us recall this result.\par
For that, we first need to define the generalized Whitehead products. The Whitehead product is a map
\[\pi_k(X)\times\pi_l(X)\to\pi_{k+l-1}(X),\]
with $X$ an arbitrary pointed topological space, 
that sends $([f],[g])$ to $[(f\vee g)\circ\varphi]$, where 
$\varphi\colon S^{k+l-1}\to S^k\vee S^l$ is the attaching map of the top-cell of the product $S^k\times S^l$. This generalizes to a map
\[[\Sigma Y,X]_\ast\times[\Sigma Y',X]_\ast\to[\Sigma(Y\wedge Y'),X]_\ast,\]
where $Y$ and $Y'$ are arbitrary pointed topological spaces, as follows. First, notice that the suspension $\Sigma Y$ is a co-H-space: it has a comultiplication map $\mu\colon \Sigma Y\to \Sigma Y\vee \Sigma Y$ and hence the set $[\Sigma Y, X]_\ast$ has a product structure given by $[f]\star[g]=[(f\vee g)\circ\mu]$, and the same happens to $[\Sigma Y',X]_\ast$. Now, the Puppe sequence associated to the cofiber sequence $Y\vee Y'\stackrel{j}{\to} Y\times Y'\stackrel{q}{\to} Y\wedge Y'$,
\[Y\vee Y'\stackrel{j}{\to} Y\times Y'\stackrel{q}{\to} Y\wedge Y'\stackrel{Qj}{\to} \Sigma(Y\vee Y') \stackrel{\Sigma j}{\to} \Sigma(Y\times Y')\stackrel{\Sigma q}{\to} \Sigma(Y\wedge Y')\to\ldots, \]
is such that $Qj$ is nullhomotopic (see \cite{puppe}). In particular, we get a short exact sequence
\[0\to[\Sigma(Y\wedge Y'),X]_\ast\to[\Sigma(Y\times Y'),X]_\ast\to[\Sigma(Y\vee Y'), X]_\ast\to0.\]
Let $\alpha\in[\Sigma Y,X]_\ast$ and $\alpha'\in[\Sigma Y',X]_\ast$ and consider the projections $\pi\colon Y\times Y'\to Y$ and $\pi'\colon Y\times Y'\to Y'$. The elements $\alpha\circ[\Sigma\pi], \alpha'\circ[\Sigma\pi']\in[\Sigma(Y\times Y'), X]_\ast$ are such that their commutator is $0$ when restricted to $\Sigma(Y\vee Y')$. Then, by exactness, there is a unique element $[\alpha,\alpha']\in[\Sigma(Y\wedge Y'),X]_\ast$ such that its restriction to $\Sigma(Y\times Y')$ gives us the aforesaid commutator. This element is the Whitehead product of $\alpha$ and $\alpha'$.\par
Take now $X=\Sigma Y\vee \Sigma Y'$. Since Whitehead products are related among themselves, there are some choices we need to make in order to define the Hilton-Milnor splitting. The basic Whitehead products of weight $1$ are $\iota\in[\Sigma Y,\Sigma Y\vee \Sigma Y']_\ast$ and $\iota'\in[\Sigma Y',\Sigma Y\vee \Sigma Y']_\ast$ represented by the inclusions
\[\Sigma Y\to \Sigma Y\vee \Sigma Y'\text{ and }\Sigma Y'\to\Sigma Y\vee\Sigma Y'.\]
Order them as $\iota<\iota'$. Let us now define the basic Whitehead products of weight $w>1$ inductively. Suppose that the basic products of weight $<w$ are defined and ordered. Then, the basic Whitehead products of weight $w$ are those of the form $[a,b]$ with $a$ a basic product of weight $u$, $b$ a basic product of weight $v$, $u+v=w$, $a<b$ and if $b=[c,d]$ then $c\leq a$. We order the basic products of weight $w$ arbitrarily among themselves and declare them greater than any basic product of lower weight.\par
For example, the only basic product of weight $2$ is $[\iota,\iota']$ and there are two basic products of weight $3$, $[\iota,[\iota,\iota']]$ and $[\iota',[\iota,\iota']]$. We order them arbitrarily among them and declare them greater than $\iota,\iota',[\iota,\iota']$. Depending on this ordering, the basic products of weight $4$ will be different.\par
Such a choice of basic Whitehead products is called a system of basic Whitehead products.
\begin{theorem}[{ \cite{hilton,milnor}}] Let $\Lambda$ be a system of basic Whitehead products. Then, there is an isomorphism
\begin{equation*} \pi_m(\Sigma Y\vee \Sigma Y') \xrightarrow{\bigoplus_{\lambda\in\Lambda}H_\lambda}\bigoplus_{\lambda\in\Lambda}\pi_m(\Sigma Y_\lambda)=\pi_m(\Sigma Y)\oplus\pi_m(\Sigma Y')\oplus\pi_m(\Sigma(Y\wedge Y'))\oplus\ldots,\end{equation*}
where $Y_\lambda$ denotes the smash product of as many copies of $Y$ (resp. $Y'$) as times as $\iota$ (resp. $\iota'$) appear in $\lambda$, and the inclusion of each of the summands is given by postcomposition with the Whitehead product $\lambda$.
\end{theorem}
Applying this to $Y=\Th(\widetilde{\theta}^\ast\gamma_{d-1})$ and $Y'=\Th(\widetilde{\theta}'^\ast\gamma_{d'-1})$ gives us the desired splitting of the homotopy groups of the Pontryagin-Thom spaces for links:
\begin{align*}
\pi_m(\Sigma \Th(\widetilde{\theta}^\ast\gamma_{d-1})\vee \Sigma \Th(\widetilde{\theta}'^\ast\gamma_{d'-1}))
\cong& \pi_m(\Sigma \Th(\widetilde{\theta}^\ast\gamma_{d-1})) \oplus\pi_m(\Sigma \Th(\widetilde{\theta}'^\ast\gamma_{d'-1}))\\
&  \oplus\pi_m(\Sigma(\Th(\widetilde{\theta}^\ast\gamma_{d-1})\wedge \Th(\widetilde{\theta}'^\ast\gamma_{d'-1})))\oplus\ldots.\end{align*}
The first two terms in the above splitting are the ones corresponding to the square  \eqref{square}. In addition to this, Wang \cite{geohilton,geohiltonthesis} defined a geometric map
\begin{align*}
\Delta_{[\iota,\iota']}\colon \mathrm{LCob}^{(\theta',\theta)}(S^m)&\to \mathrm{Cob}^{i\circ(\widetilde{\theta}'\times\widetilde{\theta})}(S^m) \\
[K\sqcup K']&\mapsto [\tau(K,K')]
\end{align*}
as follows. For $K\sqcup K'$ a $(\theta',\theta)$-link inside $S^m$, let us construct $\tau(K,K')$, a $\left(i\circ(\widetilde{\theta}'\times\widetilde{\theta})\right)$-submanifold of $S^m$ disjoint to the link. Here we have abused notation and wrote $\widetilde{\theta}'\times\widetilde{\theta}$ for the composition
\[B'\times B\to BO(d'-1)\times BO(d-1)\to BO(d'+d-2).\]
Let $W'\subseteq S^m\times[0,1]$ be a $\theta'$-cobordism starting on $K'$, intersecting $W=K\times[0,1]$ transversely and such that $K\times\{1\}$ and $W'\cap(S^m\times\{1\})$ are separated by some equator $S^{m-1}\subseteq S^m\times\{1\}$. Then, $\overline{Z}=W\cap W'$ is a $(\theta'\times\theta)$-submanifold of $S^m\times[0,1]$. In particular, since both $W$ and $W'$ have a framed normal direction, $\overline{Z}$ will have two of them. By using the fact that $W=K\times[0,1]$ is a cylinder, we will be able to isotope $\overline{Z}$ to $S^m$ and forget one of its framed normal directions. Indeed, $\overline{Z}$ has a normal vector field $v$ coming from the framed normal direction in $K\times[0,1]$. Now, we can deform $K\times[0,1]$ to $K\times[0,\varepsilon]$ and then isotope $K\times[0,\varepsilon]$ inside $S^m$ by rotating $[0,\varepsilon]$ to the framed normal direction of $K$ inside $S^m$. By doing this, the deformation of $v$ lies on the negative part (outside) of $S^m\times[0,1]$ and hence the deformation of $\overline{Z}$ is naturally a $\left(i\circ(\widetilde{\theta}'\times\widetilde{\theta})\right)$-submanifold of $S^m$. We declare this deformation $\tau(K,K')$. See \cref{tau} for a drawing and \cite[Section 4.1]{geohiltonthesis} for a more precise definition.\par
Wang's construction gives an invariant for cobordism of links.
\begin{theorem}[{ \cite{geohiltonthesis}}]\label{invariant}Let $K\sqcup K'$ be a $(\theta',\theta)$-link inside $S^m$ with $\theta$ factoring over $BO(d-1)$ and $\theta'$ factoring over $BO(d'-1)$. We have that 
\[[K\sqcup K']=0\text{ in } \mathrm{LCob}^{(\theta',\theta)}(S^m)\xRightarrow{\quad}[K],[K'],[\tau(K,K')]=0.\]
\end{theorem}

\begin{example}\label[example]{favex1} Consider the framed link inside $S^m$ consisting of a framed equator $S^{m-1}\subseteq S^m$ and two disjoint oppositely framed points, one in each hemisphere of $S^m$. See \cref{tau} for an illustration of this in the case the background sphere has dimension $m=2$.\par
  This constitutes an example of a framed link $K\sqcup K'$ of $S^m$ such that $K$ and $K'$ are framed nullbordant separately, but it is not framed nullbordant as a link. Indeed,  $[K\sqcup K']\neq0$ in $\mathrm{LCob}^{(\ast,\ast)}(S^m)$ by \cref{invariant}, since $\tau(K,K')$ is a framed point, which is not nullbordant.\par Alternatively, one could argue as follows. Assume there is a nullbordism of framed links. By gluing to it a disc and a segment along their boundaries in such a way that the disc and the segment intersect transversely in a single point (as in \cref{tau} for $m=2$), we get a closed curve and a closed hypersurface in $S^m\times[0,1]$ that intersect transversely in an odd number of points. This cannot happen since the intersection number (modulo 2) is a homotopy invariant and we can homotope our curve and our hypersurface to be disjoint. We get to a contradiction and hence such a nullbordism cannot exist.\par
\begin{figure}
    \centering
    \includegraphics[scale=0.57]{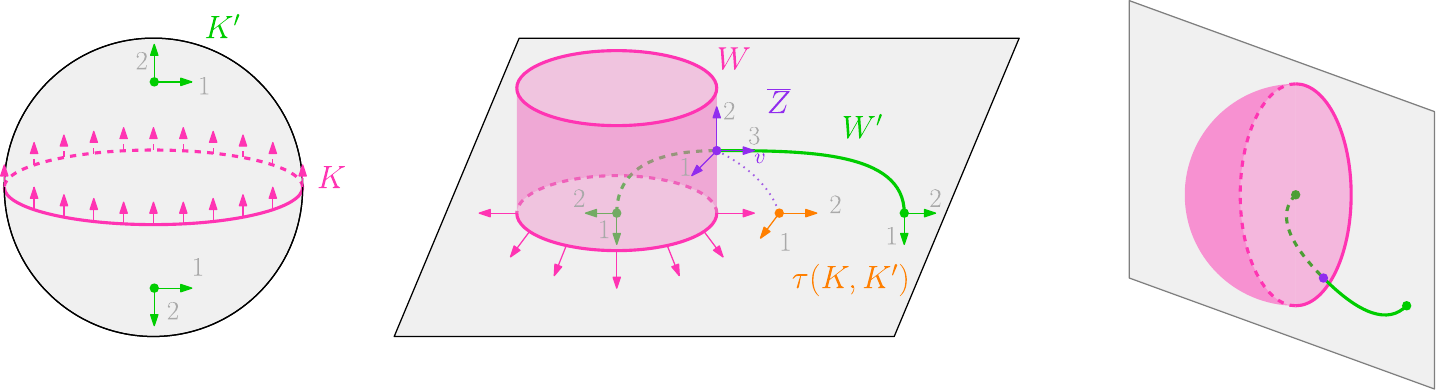}
    \caption{On the left, example of a link $K\sqcup K'$ such that both $K$ and $K'$ are nullbordant, but $K\sqcup K'$ is not nullbordant. The link consists of a pink circle $K$ framed inside $S^2$ in the direction of the pink arrows, and two green points $K'$ framed inside $S^2$ in the direction of the green arrows. In the middle, sketch of the computation of the invariant $\tau(K,K')$; since it is not nullbordant, $K\sqcup K'$ is not nullbordant. On the right, drawing of a nullbordism of $K$ and a nullbordism of $K'$ intersecting at a point.}
    \label{tau}
\end{figure}
This example proves that, in general, the $\tau$ invariant is not zero and the map
\begin{align*}
\mathrm{LCob}^{(\theta',\theta)}(M)  &\xrightarrow{(\Delta_{\iota'},\Delta_{\iota})}    \mathrm{Cob}^{\theta'}(M)\times\mathrm{Cob}^{\theta}(M)\\
        [K\sqcup K'] & \xmapsto{\hspace{1.1cm}} ([K'],[K])
\end{align*}
is not always an isomorphism. In particular, the cobordism sets of links $\mathrm{LCob}^{(\theta',\theta)}(M)$ do not necessarily split as $\mathrm{Cob}^{\theta'}(M)\times\mathrm{Cob}^{\theta}(M)$.\end{example}
Now, observe that
\[\Th\left(\left(i\circ(\widetilde{\theta}'\times\widetilde{\theta})\right)^\ast\gamma_{d+d'-1}\right)\cong\Sigma\Th\left((\widetilde{\theta}'\times\widetilde{\theta})^\ast\gamma_{d+d'-2}\right)\cong\Sigma\left(\Th(\widetilde{\theta}'^\ast\gamma_{d'-1})\wedge\Th(\widetilde{\theta}^\ast\gamma_{d-1})\right)\]
as a consequence of \cref{atiyah}. Then, for $k_1,k_2<m-1$, Wang's map fits into a square that commutes up to a sign (see \cite[Satz 3.2.2 and Satz 4.3.1]{geohiltonthesis} or \cite[Theorem 2.3]{geohilton} for the framed case):
\begin{center}
        \begin{tikzcd}      \pi_m(\Sigma\Th(\widetilde{\theta}^\ast\gamma_{d-1})\vee\Sigma\Th(\widetilde{\theta}'^\ast\gamma_{d'-1})) \arrow{r}{H_{[\iota,\iota']}} \arrow{d}{\ref{linkPT}}[swap]{\cong} &  \pi_m(\Sigma(\Th(\widetilde{\theta}^\ast\gamma_{d-1})\wedge \Th(\widetilde{\theta}'^\ast\gamma_{d'-1}))) \arrow{d}{\ref{PTconst}}[swap]{\cong} \\ \mathrm{LCob}^{(\theta',\theta)}(S^m)  \arrow{r}{\Delta_{[\iota,\iota']}}  &  \mathrm{Cob}^{i\circ(\widetilde{\theta}'\times\widetilde{\theta})}(S^m).
    \end{tikzcd}\end{center}
This means that we get a geometric interpretation of the Hilton-Milnor coefficient $H_{[\iota,\iota']}$ up to a sign.\par
Furthermore, if $\theta, \theta'\simeq\ast$, Wang \cite{geohilton} identified all Hilton maps $H_\lambda$ with geometric maps
\[ \Delta_\lambda\colon\mathrm{LCob}^{(\ast,\ast)}(S^m)\to \mathrm{Cob}^{\ast}(S^m)\]
as long as we have links of codimension higher than $1$. He hence describes a full set of invariants for a framed link inside $S^m$ of codimension larger than $1$ to be nullbordant.
\begin{theorem}[{ \cite{geohilton}}]\label[theorem]{wang}
Let $K\sqcup K'$ be a framed link inside $S^m$ of codimension larger than $1$. We have that 
\[[K\sqcup K']=0\text{ in } \mathrm{LCob}^{(\ast,\ast)}(S^m)\xLeftrightarrow{\quad}\Delta_\lambda([K\sqcup K'])=0\ \forall\lambda\in\Lambda,\]
for $\Lambda$ a system of basic Whitehead products.\end{theorem}
Wang did not explicitly find a full set of cobordism invariants of links in general when $\theta$, $\theta'$ factor over $BO(d-1)$, $BO(d'-1)$, but he claims that there are no essential difficulties to generalize the arguments used in the framed case.

\subsection{Turning nested manifolds into links}\label{subsec3.2}
Let $M$ be a closed $m$-manifold, fix integers $k_1,k_2$ with $0\leq k_2<k_1<m$ and consider structures $\theta\colon B\to BO(d)$ and $\theta'\colon B'\to BO(d')$ as in \cref{def:structure} for $d=m-k_1$ and $d'=k_1-k_2$. We will look at nested $(\theta',\theta)$-submanifolds of $M$ when $\theta$ factors over $BO(d-1)$:
\begin{center}
    \begin{tikzcd}
        B \arrow{r}{\widetilde{\theta}} \arrow[bend right]{rr}{\theta} & BO(d-1) \arrow{r}{i} & BO(d),
    \end{tikzcd}
\end{center}
that is, when the highest-dimensional manifold has a normal bundle with one framed direction.\par
Let $K'\subseteq K$ be a nested $(\theta',\theta)$-submanifold of $M$. Define its unnesting $\Upsilon(K'\subseteq K)$ as $K\sqcup K'$, where $K'$ has been displaced away from $K$ via the framed normal direction of $K$, as illustrated in \cref{fig4}. Notice that for $\theta'\times\theta$ the composition
\[B'\times B\to BO(d')\times BO(d)\to BO(d'+d),\]
$K'$ is naturally a  $(\theta'\times\theta)$-submanifold of $M$: the product of the lift of $\nu_{K'}^K$ to $B'$ and the lift of $\nu_K^M\big|_{K'}$ to $B$ gives a lift of $\nu_{K'}^M$ to $B\times B'$.\par
If $K'\subseteq K\subseteq M$ is cobordant to $\widetilde{K}'\subseteq \widetilde{K}\subseteq M$ through $W'\subseteq W\subseteq M\times[0,1]$, since $W$ is a $\theta$-submanifold of $M\times[0,1]$ and $\theta$ factors through $BO(d-1)$, $W$ also has a privileged normal direction and moving $W'$ away from $W$ via this direction gives us a cobordism from $\Upsilon(K'\subseteq K)$ to $\Upsilon(\widetilde{K}'\subseteq \widetilde{K})$. We then have a well-defined {\bf unnesting map}
\[\Upsilon\colon\mathrm{NCob}^{(\theta',\theta)}(M)\to\mathrm{LCob}^{(\theta'\times\theta,\theta)}(M).\]
This becomes a group homomorphism when the domain and the target are groups since it preserves disjoint union. We stress that we need the assumption on the structure $\theta$ for the map to be defined.\par 
  \begin{figure}
        \centering
    \includegraphics[scale=0.77]{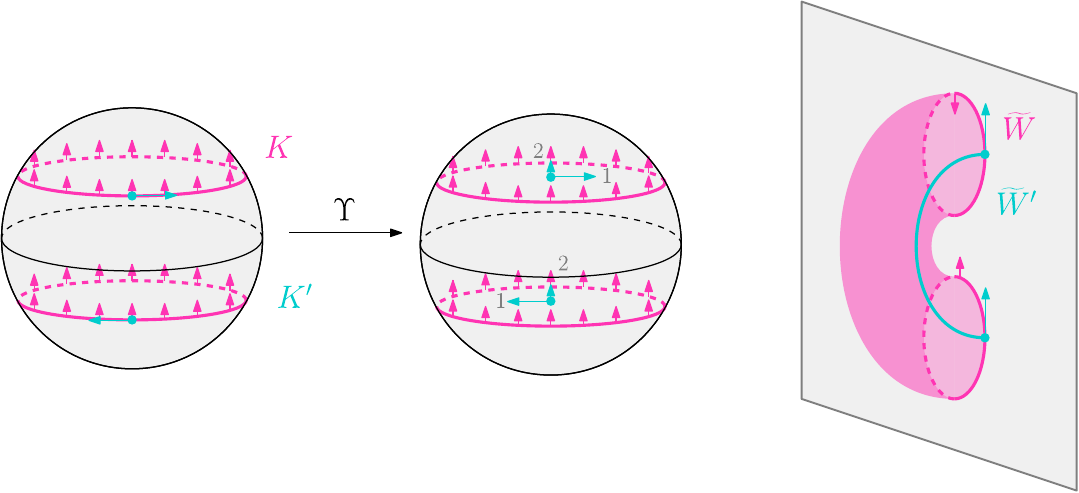}
    \caption{On the left, example of a nested submanifold $K'\subseteq K$ of $S^2$ such that both $K$ and $K'$ are nullbordant, but $K'\subseteq K$ is not nullbordant. The nested submanifold consists of two pink circles $K$ framed inside $S^2$ in the direction of the pink arrows, and two blue points $K'$ framed inside $K$ in the direction of the blue arrows. In the middle, sketch of the unnesting $\Upsilon(K'\subseteq K)$ of the nested manifold $K'\subseteq K$. On the right, drawing of a nullbordism of two circles with two points inside that does not extend to our framings. However, notice that if we flip the framing of one of our points and one of our circles, this nullbordism could be framed accordingly. This shows the importance of the framings.}
    \label{fig4}
  \end{figure} 
This subsection is devoted to proving the following theorem. 
\begin{theorem}\label[theorem]{nestsandknots} For $\theta$ factoring over $BO(d-1)$, the unnesting map
\[\Upsilon\colon\mathrm{NCob}^{(\theta',\theta)}(M)\to\mathrm{LCob}^{(\theta'\times\theta,\theta)}(M)\]
is bijective.\end{theorem}
In particular, this implies that the nested cobordism sets/groups $\mathrm{NCob}^{(\theta',\theta)}(M)$ do not necessarily split as $\mathrm{Cob}^{\theta'\times\theta}(M)\times\mathrm{Cob}^{\theta}(M)$ when $\theta$ factors over $BO(d-1)$, since the following triangle commutes:
\begin{center}
\begin{tikzcd}
    \mathrm{LCob}^{(\theta'\times\theta,\theta)}(M) \arrow{rd}{(\Delta_{\iota'},\Delta_{\iota})}  & \\
    \mathrm{NCob}^{(\theta',\theta)}(M) \arrow{r} \arrow{u}{\Upsilon} & \mathrm{Cob}^{\theta'\times\theta}(M)\times\mathrm{Cob}^{\theta}(M)\\[-3ex]
        {[K'\subseteq K]} \arrow[mapsto]{r} & {([K'],[K])}.
\end{tikzcd}
\end{center}
For our goal, we will focus on the nested Pontryagin-Thom spaces ${\Th(\theta'^\ast\gamma_{d'})}_+\wedge \Th(\theta^\ast\gamma_{d})$ from \cref{nestedcob}. As in the previous section, in this situation one of the Thom spaces involved is a suspension:
\begin{equation}\label{eq1}
\Th(\theta^\ast\gamma_{d})\cong\Sigma\Th(\widetilde{\theta}^\ast\gamma_{d-1}).
\end{equation}
Let us summarize some properties of suspensions that will be useful for us in the following lemma.
\begin{lemma}\label[lemma]{propsusp} For $(X,x_0)$ a pointed CW-complex and $X_+=(X\sqcup\ast,\ast)$, we have that:
\renewcommand{\theenumi}{\alph{enumi}}
\begin{enumerate}
\item $\Sigma(X_+)\simeq \Sigma X\vee S^1$.
\item For $\pi\colon \Sigma(X_+)\cong X_+\wedge S^1\to S^1$ the projection onto $S^1$, we have that the composition of $\pi$ with the homotopy equivalence $\Sigma X\vee S^1\stackrel{\simeq}{\to}\Sigma(X_+)$ is homotopic to the map $\Sigma X\vee S^1\to S^1$ collapsing $\Sigma X$. 
\item For $\widetilde{\pi}\colon \Sigma(X_+)\to \Sigma X$ the suspension of the pointed map $X_+\to X$ that is the identity on $X$, we have that the composition of $\widetilde{\pi}$ with the homotopy equivalence $\Sigma X\vee S^1\stackrel{\simeq}{\to}\Sigma(X_+)$ is homotopic to the map $\Sigma X\vee S^1\to \Sigma X$ collapsing $S^1$. 
\end{enumerate}
\end{lemma}
\begin{proof}
For the first claim, we will proceed along the lines of \cite[Example 0.8]{hatcher}. Consider $SX$ the unreduced suspension of $X$ and attach a segment to its extrema, obtaining a space $SX\cup_{S^0}[0,1]$. By collapsing $\{x_0\}\times [0,1]\subseteq SX$, we get $\Sigma X\vee S^1$; on the other hand, by collapsing $[0,1]$ we get $\Sigma(X_+)$. Since in both cases we are collapsing a contractible CW-subcomplex of our initial complex, we see that both $\Sigma(X_+)$ and $\Sigma X\vee S^1$ are homotopy equivalent to $SX\cup_{S^0}[0,1]$ (see e.g. \cite[page 11]{hatcher}) and hence homotopy equivalent to each other. This is illustrated in \cref{fig7}.\par
\begin{figure}
    \centering
    \includegraphics[scale=0.6]{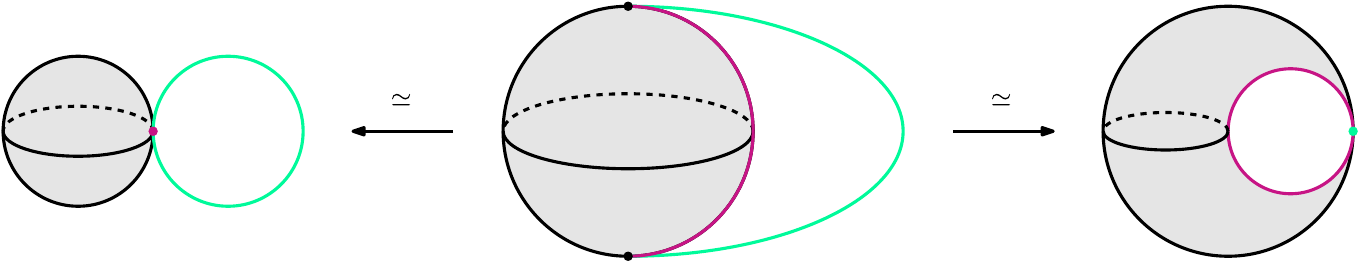}
    \caption{Drawing of $\Sigma X\vee S^1\simeq SX\cup_{S^0}[0,1]\simeq\Sigma(X_+)$ for $X=S^1$.}
    \label{fig7}
\end{figure}
The second claim follows from observing that the composition of $\pi$ with the homotopy equivalence $SX\cup_{S^0}[0,1]\stackrel{\simeq}{\to}\Sigma(X_+)$ is homotopic to the map $SX\cup_{S^0}[0,1]\to S^1$ collapsing $SX$ via the homotopy
\begin{align*}
    H\colon SX\cup_{S^0}[0,1]\times[0,1]&\to S^1\\
    ([x,s],t)&\mapsto e^{\pi it+2\pi i(1-t)s} \text{ if $[x,s]\in SX$},\\
    (s,t)&\mapsto e^{\pi it(1-2s)} \text{ if $s\in[0,1]$}.
\end{align*}
The third claim follows from observing that the composition of $\widetilde{\pi}$ with the homotopy equivalence $SX\cup_{S^0}[0,1]\stackrel{\simeq}{\to}\Sigma(X_+)$ is the collapse of $\{x_0\}\times[0,1]\cup_{S^0}[0,1]$.
\end{proof}
Now, we are ready to prove our theorem.
\begin{proof}[Proof of \cref{nestsandknots}] 
Let us first see that 
\begin{equation}\label{eqq}
{\Th(\theta'^\ast\gamma_{d'})}_+\wedge \Th(\theta^\ast\gamma_{d})\simeq\Th((\theta'\times\theta)^\ast\gamma_{d+d'})\vee\Th(\theta^\ast\gamma_{d}).\end{equation}
Indeed, we know that one of the Thom spaces involved is a suspension (see \eqref{eq1}), that a suspension is a smash product with the circle $S^1$ and that the smash product is associative and commutative. Hence, 
\[{\Th(\theta'^\ast\gamma_{d'})}_+\wedge \Th(\theta^\ast\gamma_{d})\cong{\Th(\theta'^\ast\gamma_{d'})}_+\wedge \Sigma\Th(\widetilde{\theta}^\ast\gamma_{d-1})\cong \Sigma({\Th(\theta'^\ast\gamma_{d'})}_+)\wedge \Th(\widetilde{\theta}^\ast\gamma_{d-1}).\]
In this situation, we can use \cref{propsusp} and the distributivity of the smash product over the wedge sum:
\begin{align*}
\Sigma({\Th(\theta'^\ast\gamma_{d'})}_+)\wedge \Th(\widetilde{\theta}^\ast\gamma_{d-1})&\simeq(\Sigma\Th(\theta'^\ast\gamma_{d'})\vee S^1)\wedge\Th(\widetilde{\theta}^\ast\gamma_{d-1})\\
    &\cong (\Sigma\Th(\theta'^\ast\gamma_{d'})\wedge\Th(\widetilde{\theta}^\ast\gamma_{d-1}))\vee (S^1\wedge \Th(\widetilde{\theta}^\ast\gamma_{d-1})).
\end{align*}
Next, we can use again that smashing with $S^1$ means suspending, that the smash product is commutative and that by \eqref{eq1} one of our Thom spaces is a suspension: 
\begin{align*} (\Sigma\Th(\theta'^\ast\gamma_{d'})\wedge\Th(\widetilde{\theta}^\ast\gamma_{d-1}))\vee (S^1\wedge \Th(\widetilde{\theta}^\ast\gamma_{d-1}))    &\cong (\Th(\theta'^\ast\gamma_{d'})\wedge\Sigma\Th(\widetilde{\theta}^\ast\gamma_{d-1}))\vee \Sigma\Th(\widetilde{\theta}^\ast\gamma_{d-1})\\
    &\cong(\Th(\theta'^\ast\gamma_{d'})\wedge\Th(\theta^\ast\gamma_{d}))\vee \Th(\theta^\ast\gamma_{d}).
\end{align*}
Lastly, by \cref{atiyah}, we know that the smash product of two Thom spaces is again a Thom space:
\[(\Th(\theta'^\ast\gamma_{d'})\wedge\Th(\theta^\ast\gamma_{d}))\vee \Th(\theta^\ast\gamma_{d})\cong\Th((\theta'\times\theta)^\ast\gamma_{d+d'})\vee\Th(\theta^\ast\gamma_d).\]
The composition of all of the above equivalences yields the claim. Now, the theorem is result of the commutativity of the following diagram:
\begin{center}
    \begin{tikzcd}
        \mathrm{NCob}^{(\theta',\theta)}(M) \arrow{r}{\Upsilon} \arrow{d}{\cong}[swap]{\ref{nestedcob}} & \mathrm{LCob}^{(\theta'\times\theta,\theta)}(M) \arrow{d}{\cong}[swap]{\ref{linkPT}}\\
        \left[M, {\Th(\theta'^\ast\gamma_{d'})}_+\wedge \Th(\theta^\ast\gamma_{d})\right] \arrow{r}{\cong}[swap]{\eqref{eqq}} & \left[M, \Th((\theta'\times\theta)^\ast\gamma_{d+d'})\vee\Th(\theta^\ast\gamma_{d})\right].
    \end{tikzcd}
\end{center}
Let us see it indeed commutes.\par
Take $K'\subseteq K$ a nested $(\theta',\theta)$-submanifold of $M$. The left vertical arrow sends it to a map $\varphi\colon M\to{\Th(\theta'^\ast\gamma_{d'})}_+\wedge \Th(\theta^\ast\gamma_{d})$ such that the inverse image of $\Th(\theta'^\ast\gamma_{d'})\times B$ is $K$ and the inverse image of $B'\times B$ is $K'$. Taking the inverse image of $\Th(\theta'^\ast\gamma_{d'})\times B$ under $\varphi$ is equivalent to taking the inverse image of $B$ under $p\circ \varphi$ for $p\colon {\Th(\theta'^\ast\gamma_{d'})}_+\wedge \Th(\theta^\ast\gamma_{d})\to\Th(\theta^\ast\gamma_{d})$ the projection map. Taking the inverse image of $B'\times B$ under $\varphi$ is equivalent to taking the inverse image of $B'\times B$ under $\widetilde{p}\circ \varphi$ for $\widetilde{p}\colon {\Th(\theta'^\ast\gamma_{d'})}_+\wedge \Th(\theta^\ast\gamma_{d})\to \Th(\theta'^\ast\gamma_{d'})\wedge \Th(\theta^\ast\gamma_{d})$ the projection map.\par
Consider now the unnesting $\Upsilon(K'\subseteq K)=K'\sqcup K$. The right vertical arrow sends it to a map $\psi\colon M\to \Th((\theta'\times\theta)^\ast\gamma_{d+d'})\vee\Th(\theta^\ast\gamma_{d})$ such that the inverse image of $B'\times B$ is $K'$ and the inverse image of $B$ is $K$. Taking the inverse image of $B$ under $\psi$ is equivalent to taking the inverse image of $B$ under $q\circ\psi$ for $q\colon \Th((\theta'\times\theta)^\ast\gamma_{d+d'})\vee\Th(\theta^\ast\gamma_{d})\to\Th(\theta^\ast\gamma_{d})$ the map collapsing $\Th((\theta'\times\theta)^\ast\gamma_{d+d'})$. Taking the inverse image of $B'\times B$ under $\psi$ is equivalent to taking the inverse image of $B'\times B$ of $\widetilde{q}\circ\psi$ for $\widetilde{q}\colon \Th((\theta'\times\theta)^\ast\gamma_{d+d'})\vee\Th(\theta^\ast\gamma_{d})\to\Th((\theta'\times\theta)^\ast\gamma_{d+d'})$ the map collapsing $\Th(\theta^\ast\gamma_{d})$. \par
But by \cref{propsusp}, the composition of $p$ with the homotopy equivalence 
\[\Th((\theta'\times\theta)^\ast\gamma_{d+d'})\vee\Th(\theta^\ast\gamma_{d})\simeq {\Th(\theta'^\ast\gamma_{d'})}_+\wedge \Th(\theta^\ast\gamma_{d})\] is homotopic to $q$. Analogously, the composition of $\widetilde{p}$ with the above homotopy equivalence is homotopic to $\widetilde{q}$. This proves the commutativity of the diagram.
\end{proof}
\cref{nestsandknots} says that, when $\theta$ factors over $BO(d-1)$, the only difference between the nested Pontryagin-Thom construction and the Pontryagin-Thom construction for links is that in the former we are taking nested preimages and in the latter we are taking disjoint preimages; that is, we are unnesting our nested manifolds.\par
As a consequence of \cref{nestsandknots}, we can apply Wang's invariants of \cref{wang,invariant} to nested manifolds.
\begin{corollary}\label[corollary]{wang2} Let $K'\subseteq K$ be a nested $(\theta',\theta)$-submanifold of $S^m$, with $\theta$ factoring over $BO(d-1)$. We have that 
\[[K'\subseteq K]=0\text{ in } \mathrm{NCob}^{(\theta',\theta)}(S^m)\xRightarrow{\quad} [K],[K'],\Delta_{[\iota,\iota']}(\Upsilon[K'\subseteq K])=0.\]
\end{corollary}
\begin{corollary}\label[corollary]{wang3}
Let $K'\subseteq K$ be a framed nested submanifold of $S^m$ where $K$ has codimension larger than $1$. We have that 
\[[K'\subseteq K]=0\text{ in } \mathrm{NCob}^{(\ast,\ast)}(S^m)\xLeftrightarrow{\quad}\Delta_\lambda(\Upsilon[K'\subseteq K])=0\ \forall\lambda\in\Lambda,\]
for $\Lambda$ a system of basic Whitehead products.
\end{corollary}
  \begin{example}\label[example]{favex2} 
  Consider the nested framed submanifold of $S^m$ consisting of two copies of $S^{m-1}$ with the same framing and with one framed point inside each, the points having opposite framings. See \cref{fig4} for an illustration of this in the case the background sphere has dimension $m=2$. This constitutes an example of a nested framed submanifold $K'\subseteq K$ of $S^m$ such that $K$ and $K'$ are framed nullbordant separately, but together are not nested framed nullbordant. Indeed, $\Upsilon[K'\subseteq K]$ is the framed link cobordism class of the link in \cref{favex1}, hence nonzero.
\end{example}
This example is a witness of the unstable nested cobordism sets $\mathrm{NCob}^{(\theta',\theta)}(M)$ not necessarily splitting as $\mathrm{Cob}^{\theta'\times\theta}(M)\times\mathrm{Cob}^{\theta}(M)$.
\begin{remark}
To see that the nested manifold in \cref{fig4} is not nullbordant, one could also argue directly, without involving links, as follows. Assume that there exists a nested framed cobordism $W'\subseteq W$ inside $S^2\times[0,1]$ with
\[\partial W'=W'\cap(S^2\times\{0\})=K',\ \partial W=W\cap(S^2\times\{0\})=K.\]
Now, there exists a submanifold $\widetilde{W}'\subseteq \widetilde{W}$ of $S^2\times[-1,0]$ such that $\widetilde{W}'$ is a segment, $\widetilde{W}$ is a cylinder, and
\[ \partial \widetilde{W}'=\widetilde{W}'\cap(S^2\times\{0\})=K',\ \partial \widetilde{W}=\widetilde{W}\cap(S^2\times\{0\})=K\]
if we forget about the framings (as on the right of \cref{fig4}).
   Let us now glue $\widetilde{W}'\subseteq\widetilde{W}$ and $W'\subseteq W$ along $K'\subseteq K$.    
   In particular, we get a closed surface $\widetilde{W}\cup_{K}W$ embedded in $S^2\times[-1,1]$ that is non-orientable since a tubular neighbourhood of $\widetilde{W}'\cup_{K'}W'$ constitutes an embedding of a M\"{o}bius strip. We hence get to a contradiction because non-orientable closed surfaces cannot be embedded in $\mathbb{R}^3$ (see e.g. \cite{samelson}). Therefore, no nested framed nullbordism $W'\subseteq W$ exists. \par
\end{remark}
\begin{remark} \cref{nestsandknots} also explains why stable nested cobordism groups split in this case: the unnesting map is compatible with the inclusions $S^m\hookrightarrow S^{m+1}$, and stably, we can make our links disjoint enough so that as long as each component is nullbordant, the link is nullbordant as well. Besides, this viewpoint illustrates that wedges of suspensions of Thom spaces are stably equivalent to the product of those.
\end{remark}

\subsection{Non-linked example}\label{subsec3.3}
\cref{wall} tells us that nested cobordism groups split stably, that is, when we consider nested manifolds without a fixed background manifold. However, in \cref{subsec3.2} we have seen that when we consider nested submanifolds of spheres where the highest-dimensional submanifold has a normal bundle with a framed direction, the unstable nested cobordism sets have interesting geometric invariants that obstruct the splitting. We do this via unnesting our nested manifolds and turning them into links. In the case there is no such framed direction, there is no hope of having a well-defined unnesting map nor any other bijection between nested cobordism sets and cobordism sets of links, as the following example shows.
\begin{figure}
    \centering
    \includegraphics[scale=0.7]{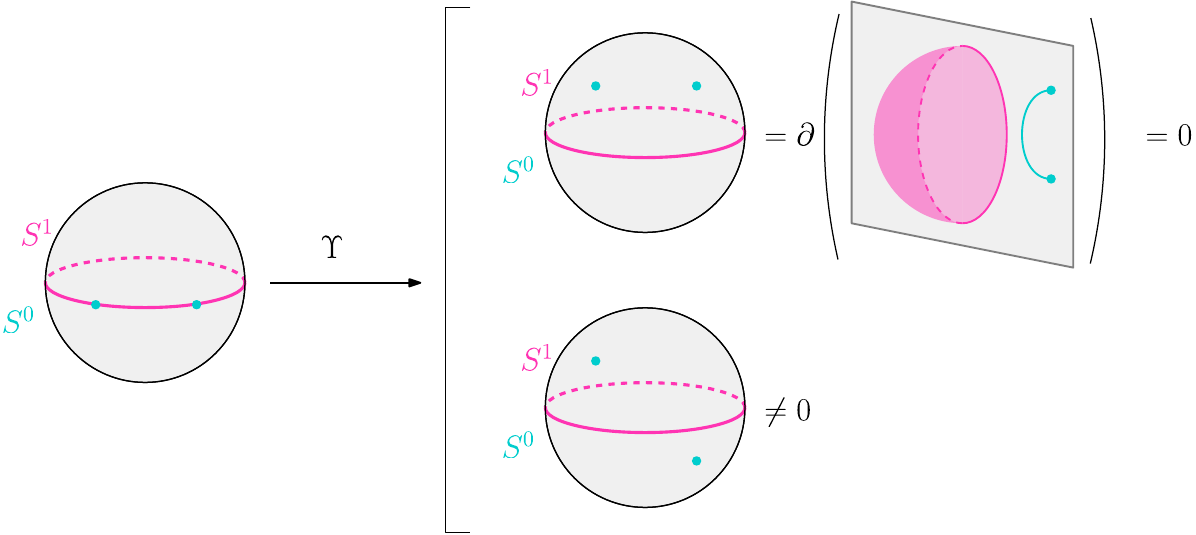}
    \caption{The unnesting map is not well-defined when we do not have a framed normal direction on our highest-dimensional manifold: there are two options for unnesting $S^0\subseteq S^1\subseteq S^2$: one of them is nullbordant and the other one is not.}
    \label{newfig}
\end{figure}
\begin{example} For $\theta=\mathrm{id}\colon BO(1)\to BO(1)$, $\theta'=\mathrm{id}\colon BO(m-1)\to BO(m-1)$, consider the nested $(\theta',\theta)$ -submanifold of $S^m$ consisting on the equator $S^{m-1}$ with two points $S^0$ inside, as in \cref{newfig} for $m=2$. Then, there are essentially two options for its unnesting $\Upsilon(S^0\subseteq S^{m-1})=S^0\sqcup S^{m-1}$:
\begin{itemize}
    \item Both points $S^0$ lie on the same hemisphere of $S^m$, in which case $S^0\sqcup S^{m-1}$ is nullbordant since it bounds the disjoint union of a segment and a disc.
    \item Each point of $S^0$ lies on a different hemisphere of $S^m$, in which case $S^0\sqcup S^{m-1}$ is not nullbordant by a similar argument to that of \cref{favex1}.
\end{itemize}
In fact, we do not only not have an unnesting map, but the nested cobordism sets and the cobordism sets of links do not coincide in this case:
\begin{itemize}
    \item $\mathrm{NCob}^{(\theta',\theta)}(S^2)$ has two elements. Indeed, a nested $(\theta',\theta)$-submanifold of $S^2$ consists of a disjoint union of circles with points inside; circles with an even number of points are nullbordant, circles with an odd number of points are cobordant among them and two circles with one point inside each are nullbordant.
    \item $\mathrm{LCob}^{(\theta'\times\theta,\theta)}(S^2)$ has at least three elements. Indeed, the empty link, the link consisting of a point and an empty $1$-manifold and the link in the bottom right of \cref{newfig} are three $(\theta'\times\theta,\theta)$-links inside $S^2$ that are pairwise not cobordant and hence represent three different classes in $\mathrm{LCob}^{(\theta'\times\theta,\theta)}(S^2)$.
\end{itemize}\end{example}
Even though this situation is more puzzling, we still know that the unstable nested cobordism sets/groups do not necessarily split in this case as well. Here is an example of the nonsplitting of the unstable nested cobordism sets even when the highest-dimensional manifold does not have a normal bundle with a framed direction.
\begin{example}\label[example]{favex3} For $\theta=\mathrm{id}\colon BO(1)\to BO(1)$ and $\theta'\simeq\ast\colon\ast\to BO(1)$, the $(\theta',\theta)$-submanifold $K'\subseteq K$ of $S^2$ in \cref{fig6} is such that $K'$ is a nullbordant $(\theta'\times\theta)$-submanifold of $S^2$ and $K$ is a nullbordant $\theta$-submanifold of $S^2$. However, $K'\subseteq K$ is not nested $(\theta',\theta)$-nullbordant.\par
\begin{figure}
    \centering
    \includegraphics[scale=0.8]{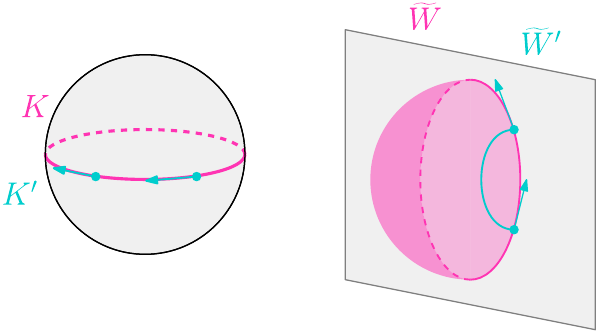}
    \caption{On the left, example of a nested submanifold $K'\subseteq K$ of $S^2$ such that both $K$ and $K'$ are nullbordant, but $K'\subseteq K$ is not nullbordant. The nested submanifold consists of a pink circle $K$ and two blue points $K'$ framed inside $K$ in the direction of the blue arrows. On the right, drawing of a nullbordism of a circle with two points inside that does not extend to our framings.}
    \label{fig6}
\end{figure}
Indeed, assume that there exists a nested nullbordism, that is, a $(\theta',\theta)$-cobordism $W'\subseteq W$ inside $S^2\times[0,1]$ with
\[\partial W'=W'\cap(S^2\times\{0\})=K',\ \partial W=W\cap(S^2\times\{0\})=K.\]
Now, choose manifolds $\widetilde{W}'\subseteq \widetilde{W}$ inside $S^2\times[-1,0]$ such that $\widetilde{W}'$ is a segment, $\widetilde{W}$ is a disc, and
\[ \partial \widetilde{W}'=\widetilde{W}'\cap(S^2\times\{0\})=K',\ \partial \widetilde{W}=\widetilde{W}\cap(S^2\times\{0\})=K\]
if we forget about the framings (as on the right of \cref{fig6}).
   Let us now glue $\widetilde{W}'\subseteq\widetilde{W}$ and $W'\subseteq W$ along $K'\subseteq K$.    
   In particular, we get a closed surface $\widetilde{W}\cup_{K}W$ embedded in $S^2\times[-1,1]$ that is non-orientable since a tubular neighbourhood of $\widetilde{W'}\cup_{K'}W'$ constitutes an embedding of a M\"{o}bius strip. We hence get to a contradiction because closed non-orientable surfaces cannot be embedded in $\mathbb{R}^3$ (see e.g. \cite{samelson}). Therefore, no nested nullbordism $W'\subseteq W$ exists.\end{example}

\newcommand{\etalchar}[1]{$^{#1}$}
\providecommand{\bysame}{\leavevmode\hbox to3em{\hrulefill}\thinspace}
\providecommand{\MR}{\relax\ifhmode\unskip\space\fi MR }
\providecommand{\MRhref}[2]{%
  \href{http://www.ams.org/mathscinet-getitem?mr=#1}{#2}
}
\providecommand{\href}[2]{#2}


\begin{thebibliography}{GMTW09}

\bibitem[Ada74]{spectra}
J.~Frank Adams, \emph{Stable homotopy and generalised homology}, Chicago Lectures in Mathematics, University of Chicago Press, Chicago, Ill.-London, 1974.

\bibitem[Ati61a]{atiyahbord}
Michael~F. Atiyah, \emph{Bordism and cobordism}, Proc. Cambridge Philos. Soc. \textbf{57} (1961), 200--208.

\bibitem[Ati61b]{atiyah}
\bysame, \emph{Thom complexes}, Proc. London Math. Soc. (3) \textbf{11} (1961), 291--310.

\bibitem[Aya09]{ayala}
David Ayala, \emph{Geometric cobordism categories}, ProQuest LLC, Ann Arbor, MI, 2009, Thesis (Ph.D.)--Stanford University.

\bibitem[CF64]{connerfloyd}
Pierre~E. Conner and Edwin~E. Floyd, \emph{Differentiable periodic maps}, Ergebnisse der Mathematik und ihrer Grenzgebiete, (N.F.), vol. Band 33, Springer-Verlag, Berlin-G\"ottingen-Heidelberg; Academic Press, Inc., Publishers, New York, 1964.

\bibitem[CHM{\etalchar{+}}25]{cyl}
Maxine~E. Calle, Renee~S. Hoekzema, Laura Murray, Natalia Pacheco-Tallaj, Carmen Rovi, and Shruthi Sridhar-Shapiro, \emph{Nested cobordisms, {C}yl-objects and {T}emperley-{L}ieb algebras}, Topology Appl. \textbf{376} (2025), Paper No. 109448, 37.

\bibitem[FHT25]{freedhopkins}
Daniel~S. Freed, Michael~J. Hopkins, and Constantin Teleman, \emph{Discrete quantum systems from topological field theory}, 2025, arXiv:2506.05131.

\bibitem[GMTW09]{galatius2009homotopy}
S{\o}ren Galatius, Ib~Madsen, Ulrike Tillmann, and Michael Weiss, \emph{The homotopy type of the cobordism category}, Acta Math. \textbf{202} (2009), no.~2, 195--239.

\bibitem[GRW14]{stablemoduli}
S{\o}ren Galatius and Oscar Randal-Williams, \emph{{Stable moduli spaces of high-dimensional manifolds}}, Acta Mathematica \textbf{212} (2014), no.~2, 257 -- 377.

\bibitem[Hat02]{hatcher}
Allen Hatcher, \emph{Algebraic topology}, Cambridge University Press, Cambridge, 2002.

\bibitem[Hil55]{hilton}
Peter~J. Hilton, \emph{On the homotopy groups of the union of spheres}, J. London Math. Soc. \textbf{30} (1955), 154--172.

\bibitem[Hoe18]{Renee}
Renee~S. Hoekzema, \emph{Algebraic topology of manifolds}, Ph.D. thesis, University of Oxford, 2018.

\bibitem[HS65]{haefligersteer}
Andr\'{e} Haefliger and Brian Steer, \emph{Symmetry of linking coefficients}, Comment. Math. Helv. \textbf{39} (1965), 259--270.

\bibitem[Kom86]{komiya}
Katsuhiro Komiya, \emph{Cutting and pasting of pairs}, Osaka J. Math. \textbf{23} (1986), no.~3, 577--584.

\bibitem[Kos93]{kosinski}
Antoni~A. Kosinski, \emph{Differential manifolds}, Pure and Applied Mathematics, vol. 138, Academic Press, Inc., Boston, MA, 1993.

\bibitem[Las63]{lashof}
Richard~K. Lashof, \emph{Poincar\'e{} duality and cobordism}, Trans. Amer. Math. Soc. \textbf{109} (1963), 257--277.

\bibitem[Mil72]{milnor}
John~W. Milnor, \emph{On the construction {FK}}, J. F. Adams and G.C. Shepherd (ed.), Algebraic Topology: A Student’s Guide, London Mathematical Society Lecture Note Series, Cambridge University Press, 1972, p.~118–136.

\bibitem[MS74]{milnorstasheff}
John~W. Milnor and James~D. Stasheff, \emph{Characteristic classes}, Annals of Mathematics Studies, vol. No. 76, Princeton University Press, Princeton, NJ; University of Tokyo Press, Tokyo, 1974.

\bibitem[Pon55]{pontrjagin2007smooth}
Lev~S. Pontryagin, \emph{{\cyr Gladkie mnogoobraziya i ikh primeneniya v teorii gomotopi\u{\i}}}, Izdat. Akad. Nauk SSSR, Moscow, 1955, Trudy Mat. Inst. Steklov. no. 45.

\bibitem[Pup58]{puppe}
Dieter Puppe, \emph{Homotopiemengen und ihre induzierten {A}bbildungen. {I}}, Math. Z. \textbf{69} (1958), 299--344.

\bibitem[Sam69]{samelson}
Hans Samelson, \emph{Orientability of hypersurfaces in {$R\sp{n}$}}, Proc. Amer. Math. Soc. \textbf{22} (1969), 301--302.

\bibitem[Spa66]{spanier}
Edwin~H. Spanier, \emph{Algebraic topology}, McGraw-Hill Book Co., New York-Toronto-London, 1966.

\bibitem[Sto66]{cobofmaps}
Robert~E. Stong, \emph{Cobordism of maps}, Topology \textbf{5} (1966), 245--258.

\bibitem[Sto71]{strong1971cobordism}
\bysame, \emph{On the cobordism of pairs}, Pacific J. Math. \textbf{38} (1971), 803--816.

\bibitem[Tho54]{Thom}
Ren\'{e} Thom, \emph{Quelques propri\'{e}t\'{e}s globales des vari\'{e}t\'{e}s diff\'{e}rentiables}, Comment. Math. Helv. \textbf{28} (1954), 17--86.

\bibitem[Vli25]{Rolf}
Rolf~A. Vlierhuis, \emph{Cutting and pasting pairs of manifolds with tangential structures}, 2025, arXiv: 2506.15204.

\bibitem[Wal61]{Wallpairs}
Charles T.~C. Wall, \emph{Cobordism of pairs}, Comment. Math. Helv. \textbf{35} (1961), 136--145.

\bibitem[Wal16]{Wallbook}
\bysame, \emph{Differential topology}, Cambridge Studies in Advanced Mathematics, vol. 156, Cambridge University Press, Cambridge, 2016.

\bibitem[Wan98]{geohiltonthesis}
Jianhua Wang, \emph{{E}ine {G}eometrische {I}nterpretation {G}ewisser {H}ilton-{K}oeffizienten}, {P}h.{D}. thesis, Universit\"{a}t Siegen, 1998.

\bibitem[Wan04]{geohilton}
\bysame, \emph{The geometry of the {H}ilton splitting}, Topology Appl. \textbf{141} (2004), no.~1-3, 105--124.

\end{thebibliography}
\end{document}